\documentclass{article}

\usepackage[english]{babel}

\usepackage[letterpaper,top=2cm,bottom=2cm,left=3cm,right=3cm,marginparwidth=1.75cm]{geometry}

\usepackage{amsmath,amssymb,amsfonts,graphicx,amsthm,mathtools,nicefrac}
\usepackage{graphicx}
\usepackage[colorlinks=true, allcolors=blue]{hyperref}
\usepackage{multirow}
\usepackage[ruled]{algorithm2e}
\usepackage{bbding}
\usepackage{makecell}

\usepackage{subfigure}
\usepackage{parskip}
\usepackage{threeparttable}

\allowdisplaybreaks
\newcommand\numberthis{\addtocounter{equation}{1}\tag{\theequation}}

\theoremstyle{plain}
\newtheorem{theorem}{Theorem}
\newtheorem{lemma}{Lemma}
\newtheorem{assumption}{Assumption}
\newtheorem{corollary}{Corollary}

\theoremstyle{definition}
\newtheorem{remark}{Remark}
\usepackage{authblk}
\title{A Zeroth-Order Variance-Reduced Method for \\ Decentralized Stochastic Non-convex Optimization}
\author[1]{Hongxu Chen}
\author[2]{Jinchi Chen}
\author[1]{Ke Wei}
\affil[1]{School of Data Science, Fudan University, Shanghai, China.}
\affil[2]{School of Mathematics, East China University of Science and Technology, Shanghai, China.}

\begin{document}
\maketitle

\begin{abstract}
In this paper, we consider a distributed stochastic non-convex optimization problem, which is about minimizing a sum of $n$ local cost functions over a network with only zeroth-order information. 
A novel single-loop Decentralized Zeroth-Order Variance Reduction algorithm, called DZOVR, is proposed, which combines two-point gradient estimation, momentum-based variance reduction technique, and gradient tracking.
Under mild assumptions, we show that the algorithm is able to achieve $\mathcal{O}(dn^{-1}\epsilon^{-3})$ sampling complexity at each node to reach an $\epsilon$-accurate stationary point and also exhibits network-independent and linear speedup properties. 
To the best of our knowledge, this is the first stochastic decentralized zeroth-order algorithm that achieves this sampling complexity.
Numerical experiments demonstrate that DZOVR outperforms the other state-of-the-art algorithms and has network-independent and linear speedup properties.
\end{abstract}

\section{Introduction}
Distributed optimization plays an important role in multi-agent control and has been applied in diverse domains, including network sensing \cite{yang2019}, power systems \cite{guo2016,molzahn2017,wang2017}, and multi-agent reinforcement learning \cite{chen2022,zhang2021multi,zhang2018fully}.
It has received intensive investigations recently due to the challenges in tackling large-scale computing problems.
In this paper, we focus on the decentralized setting over a network.
More precisely, let $\mathcal{G}=(\mathcal{V},\mathcal{E})$ be an undirected graph where $\mathcal{V}=\{1,\dots,n\}$ is the set of nodes and $\mathcal{E}$ is the collection of edges. 
If node $i$ can communicate with node $j$, then $(i,j)\in\mathcal{E}$.
The neighbor of node $i$ is defined by $\mathcal{N}(i) = \{j \in \mathcal{N} \mid(i, j) \in \mathcal{E} \text { or } i=j\}$.
The problem can be expressed as 
\begin{equation}
    \label{eq_problem}
    \min_{x\in \mathbb{R}^d} f(x) := \frac{1}{n} \sum_{i=1}^n f_i (x),
\end{equation}
where $f_i$ is the local function of node $i$ and each node can only communicate with its neighbors.

In the past decade, many effective first-order methods have been proposed to solve problem~\eqref{eq_problem}. 
Decentralized gradient descent method (DGD) \cite{nedic2009} is a direct extension of gradient descent (GD), where each node minimizes its own objective function using GD and conducts consensus through communication. 
The technique of gradient tracking has been introduced in \cite{di2016,qu2017,scutari2019,pu2021} in order for the algorithms to achieve a convergence rate that is comparable to the centralized setting without the assumption of bounded dissimilarity.
For the stochastic non-convex problems, the convergence rate $\mathcal{O}(1/\sqrt{nK})$ of the method with gradient tracking has been established in \cite{xin2021}.
There is another line of research which develops the algorithms for problem~\eqref{eq_problem} by reformulating it as a linear constrained problem over the network \cite{liang2019,xu2018,yi2021,yi2022primal}.
When $f_i$ is specified as the expected loss function, gradient descent is usually replaced by stochastic gradient descent to minimize the local function. 
In this scenario, a variety of variance reduction techniques can be used to improve the convergence of the algorithms.
For instance, D-GET in \cite{sun2020} improves sample and communication complexity through variance reduction and gradient tracking. 
SPIDER-SFO \cite{fang2018} in the centralized setting has been extended to the decentralized scenario in \cite{pan2020}.
A single-loop distributed variance reduction method, which reduces the oracle queries per iteration and achieves linear speedup, is introduced in \cite{xin2021hybrid}.
The convergence rates of the methods in \cite{sun2020} and \cite{pan2020} are both $\mathcal{O}(1/(K)^{2/3})$, while the method in \cite{xin2021hybrid} achieves a faster convergence rate of $\mathcal{O}(1/(nK)^{2/3})$.
It's beyond the scope of this paper to give an exhaustive literature review on this topic. We refer interested readers to the general framework in \cite{xin2020general} and references therein for more details. 

However, in many practical scenarios, gradient information is not available, and we can only have access to function values, such as black-box models \cite{papernot2017,chen2017zoo,ilyas2018}. 
To handle this problem, gradient is approximated through random sampling and finite differences in zeroth-order optimization.
The fundamental properties of 2-point estimator is investigated in \cite{nesterov2017} and the convergence rate of the algorithm based on the 2-point estimator is established under convex setting, laying the foundation for a set of subsequent works.
Further convergence analysis has been conducted in \cite{ghadimi2013} for the stochastic non-convex objective functions. 
The 2-point estimator approximates the gradient by taking difference in only one direction, resulting in a high variance.
To reduce the variance, many variance reduction methods and mini-batch sampling are used in the algorithm development \cite{liu2018svrg,liu2018,liu2018vr}. 
However, mini-batch sampling requires more queries in each iteration and is less efficient compared to the dimension-dependent deterministic methods when the number of samples is large.

For the distributed problem~\eqref{eq_problem}, it is also natural to develop zeroth-order optimization methods when the gradient information is missing.
A distributed zeroth-order algorithm to solve the non-convex problem is proposed in \cite{hajinezhad2019} based on the augmented Lagrangian function. 
Two algorithms using 2-point estimator and 2$d$-point estimator are proposed in \cite{tang2020}, one of which utilizes the technique of gradient tracking. 
The convergence rates of the two algorithms, in the deterministic setting, have been shown to be $\mathcal{O}(\sqrt{d} \log K/\sqrt{K})$ and $\mathcal{O}(1/K)$, respectively.
In \cite{mhanna2022}, an algorithm which combines one-point estimate and gradient tracking is developed, achieving a convergence rate of $\mathcal{O}(1/\sqrt{K})$ in the strongly convex setting when the diminishing step size is used.
It is observed in \cite{maritan2023} that the second-order information can be utilized by the addition of just one extra point to the 2$d$-point estimator and the linear convergence rate has been established under strong convexity.
In the non-convex setting, several algorithms based on the stochastic coordinate methods \cite{zhang2021,zhang2022} and the primal-dual approaches \cite{zhang2021,yi2022} are introduced, and the $\mathcal{O}(\sqrt{d}/\sqrt{nK})$ convergence rate has been established.
Compared to first-order methods, zeroth-order methods have high variance due to random sampling for the approximation of gradient. 
As mentioned above, many variance reduction methods have been developed for the distributed first-order methods. 
However, in the field of zeroth-order distributed optimization where variance reduction is even more crucial, relevant research remains unexplored, to the best of our knowledge.
The goal of this paper is to design a distributed zeroth-order method with variance reduction to mitigate the impact of variance and achieve faster convergence.

\begin{table}[h]
\label{table}
\caption{Convergence results of related distributed zeroth-order methods.}
\vspace{10pt}
\centering
\renewcommand\arraystretch{1.5}
\begin{threeparttable}
\begin{tabular}{cccccc}
    \hline
    Method & Nonconvex & Stochastic & \makecell[c]{Bounded \\ Dissimilarity\tnote{1}} & \makecell[c]{Convergence \\Rate} &  \makecell[c]{Sampling \\Complexity\tnote{2}}\\
    \hline
    ZONE \cite{hajinezhad2019} & \Checkmark & \Checkmark & Strong & $\mathcal{O}(\frac{1}{K})$ & $\mathcal{O}(K\epsilon^{-2})$\\
    2-point DGD\cite{tang2020} & \Checkmark & \XSolidBrush & Strong & $\mathcal{O}(\frac{\sqrt{d} \log K}{\sqrt{K}})$ &  $\mathcal{O}(d\epsilon^{-4})$\\
    2d-point DGT \cite{tang2020} & \Checkmark & \XSolidBrush & No & $\mathcal{O}(\frac{1}{K})$ &  $\mathcal{O}(d\epsilon^{-2})$\\
    1P-DSGT \cite{mhanna2022} & \XSolidBrush & \Checkmark & Strong & $\mathcal{O}(\frac{1}{\sqrt{K}})$ &  $\mathcal{O}(\epsilon^{-4})$ \\
    ZO-JADE \cite{maritan2023} & \XSolidBrush & \XSolidBrush & No & $\mathcal{O}\left( {(1-c \gamma_r)^K} \right)$ & $\mathcal{O}(d\log\left(\frac{1}{\epsilon}\right))$\\
    ZODIAC \cite{zhang2021} & \Checkmark & \Checkmark & Yes & $\mathcal{O}(\frac{\sqrt{d}}{\sqrt{K}})$ & $\mathcal{O}(d\epsilon^{-4})$\\
    ZOOM \cite{zhang2022} & \Checkmark & \Checkmark & Yes & $\mathcal{O}(\frac{\sqrt{d}}{\sqrt{nK}})$ & $\mathcal{O}(dn^{-1}\epsilon^{-4})$\\
    ZODPA \cite{yi2022} & \Checkmark & \Checkmark & Weak & $\mathcal{O}(\frac{\sqrt{d}}{\sqrt{nK}})$ & $\mathcal{O}(dn^{-1}\epsilon^{-4})$\\
    ZODPDA \cite{yi2022} & \Checkmark & \Checkmark & Weak & $\mathcal{O}(\frac{\sqrt{d}}{\sqrt{nK}})$ & $\mathcal{O}(dn^{-1}\epsilon^{-4})$\\
    \textbf{DZOVR} & \Checkmark & \Checkmark & Weak & $\mathcal{O}((\frac{d}{nK})^{2/3})$ & $\mathcal{O}(dn^{-1}\epsilon^{-3})$\\
    \hline
\end{tabular} 
\begin{tablenotes}
\footnotesize
\item[1] For bounded dissimilarity, ``Strong," ``Yes," and ``Weak" respectively represent $f_i$ being Lipschitz, the standard bounded dissimilarity assumption (i.e., $\| \nabla f_i(x) - \nabla f(x)\| \le C$), and Assumption \ref{ass4}.
\item[2] The sampling complexity refers to the number of queries required for each node to reach an $\epsilon$-stationary point, i.e., $\| \nabla f(x)\| \le \epsilon$.
\end{tablenotes}
\end{threeparttable}
\end{table} 

\subsection{Contributions}
The main contributions of this paper are as follows:
\begin{itemize}
    \item A distributed zeroth-order method called DZOVR is proposed. This method combines a momentum-based technique with gradient tracking, which effectively reduces the variance of the 2-point estimator and thus can achieve faster convergence. 
    Further numerical experiments demonstrate that the proposed algorithm outperforms the other state-of-the-art methods.
    \item We prove that DZOVR converges converges to at a rate of $\mathcal{O}((d/nK)^{2/3})$ under certain conditions, or equivalently, an $\epsilon$-accurate stationary point can be reached under the $\mathcal{O}(dn^{-1}\epsilon^{-3})$ sampling complexity. 
    To the best of our knowledge, this is the first result achieving this sampling complexity in the decentralized zeroth-order stochastic non-convex optimization.
    It is worth noting that the sampling complexity only looses a $d$ factor compared to the best complexity that is achievable for the distributed first-order methods. 
    We have summarised the convergence results of related distributed zeroth-order methods in Table \ref{table}.
\end{itemize}

\subsection{Notation}
We use $\|\cdot\|$ to denote the Euclidean norm of a vector or the spectral norm of a matrix.
    The closed unit ball in $\mathbb{R}^d$ is denoted by $\mathbb{B}^d := \{ x\in \mathbb{R}^d: \| x \| \le 1 \}$ and the unit sphere is denoted by $\mathbb{S}_{d-1} := \{ x\in \mathbb{R}^d: \| x \| = 1 \}$. 
    We let $\mathcal{U}(\mathbb{R}^d)$ and $\mathcal{U}(\mathbb{S}_{d-1})$ denote the uniform distributions over $\mathbb{R}^d$ and $\mathbb{S}_{d-1}$, respectively.
    Suppose $A \in \mathbb{R}^{p \times q}$ and $B \in \mathbb{R}^{r \times s}$. Then the Kronecker product, denoted $\otimes$, is given by
    $$
    A \otimes B=\left[\begin{array}{ccc}
    a_{11} B & \cdots & a_{1 q} B \\
    \vdots & \ddots & \vdots \\
    a_{p 1} B & \cdots & a_{p q} B
    \end{array}\right] \in \mathbb{R}^{p r \times q s}.
    $$
Given a network $\mathcal{G}=(\mathcal{V},\mathcal{E})$, the corresponding consensus matrix is denoted $W$, where $w_{ij} > 0$ if $(i, j) \in \mathcal{E}$ or $i = j$, and $w_{ij} = 0$ otherwise.
    Letting $\xi^k = \{\xi_1^k, \xi_2^k, \dots, \xi_n^k\}$ be a set of independent random variables and $u^k = \{u_1^k, u_2^k, \dots, u_n^k\}$ be a set of random vectors,
    we define $\mathcal{F}_k$ as the $\sigma$-algebra generated by $\{ \xi^0, u^0, \xi^1, u^1, \dots, \xi^{k-1}, u^{k-1} \}$.

\subsection{Outline}
The rest of the paper is organized as follows.
In Section \ref{sec2}, we introduce the DZOVR algorithm and analyze its convergence.
In Section \ref{sec3}, we validate through numerical experiments that the algorithm achieves the state-of-the-art performance with linear speedup and network-independent properties.
The paper is concluded in the Section \ref{sec4}.

\section{Main Results}
\label{sec2}
\subsection{Preliminaries}
\noindent\textbf{Zeroth-order estimators.}
Given a function $f: \mathbb{R}^d \to \mathbb{R}$, according to the definition of gradient, the most natural way to estimate it is by finite differences,
$$
\hat{\nabla} f(x)_{(2 d)}:=\sum_{i=1}^d \frac{f\left(x+t e_i\right)-f(x)}{t} e_i,
$$
where $e_i \in \mathbb{R}^d$ is the $i$-th unit vector, and $t>0$ is a given constant.
Although the estimation error can be arbitrarily small, the 2$d$-point estimator requires $2d$ queries in every iteration, leading to high computational cost.
A way to deal with the problem is using random sampling, which yields the \textit{$2$-point estimator}:
$$\hat{\nabla} f(x; u):=d \; \frac{f(x+t u)-f(x- tu)}{2t} u,$$
where $u \sim \mathcal{U}(\mathbb{S}_{d-1})$ is a random perturbation.

While making the queries of estimator dimension-independent, the random estimator introduces high variance, which is in the order of $\mathcal{O}(d \| \nabla f(x) \|^2)$. 
In the centralized case, the gradient converges to zero, so does the gradient estimation variance.
However, in the distributed case, the gradients of some nodes may not converge to zero, leading to persistent high estimation variance, which can impede the convergence of the algorithm \cite{tang2020}.

\noindent\textbf{Gradient tracking.}
Decentralized gradient descent (DGD) is a simple and effective distributed optimization algorithm that can be written in the following form:
$$x^{k+1}_i = \sum_{j=1}^n w_{ij} \left( x_j^k - \alpha_k \nabla f_j(x_j^k)\right),$$
where $\alpha_k$ is the step size at the $k$-th iteration.
It is worth noting that stationary point for problem \eqref{eq_problem} is $\nabla f(x) = \frac{1}{n}\sum_{i=1}^n \nabla f_i(x) = 0$, but the gradient at each node is not necessarily equal to zero. 
Therefore, to ensure the convergence of DGD, bounded dissimilarity or diminishing step size is required.

The gradient tracking technique \cite{scutari2019,pu2021} addresses this issue by additional communication of gradients, which allows gradient consensus.
The update procedure is as follows:
\begin{align*}
g_i^{k+1} & =\sum_{j=1}^n w_{i j}\left(g_j^k+ \nabla f_j(x_j^k)-\nabla f_j(x_j^{k-1})\right), \\
x_i^{k+1} & =\sum_{j=1}^n w_{i j}\left(x_j^k-\alpha_k g_j^{k+1}\right).
\end{align*}
Under Assumption \ref{ass1}, gradient tracking possesses a crucial property that we will be used frequently in the sequel, that is,
$$\frac{1}{n} \sum_{i=1}^n g_i^{k+1} = \frac{1}{n} \sum_{i=1}^n \nabla f_i(x_i^k).$$

\noindent\textbf{Variance reduction.} 
In the stochastic setting, due to the impact of variance, first-order methods like stochastic gradient descent can only take small stepsizes, leading to slow convergence rates. 
Many variance reduction methods have been proposed to improve the convergence rate of stochastic gradient descent in recent years, such as SVRG \cite{johnson2013}, SARAH \cite{nguyen2017}, SPIDER \cite{fang2018} and STORM \cite{cutkosky2019}.
Among them, many are double-loop algorithms that require a large batch size to estimate gradients, potentially posing practical challenges in real-world applications.
Therefore, we consider the single-loop momentum-based variance reduction method proposed in \cite{cutkosky2019}, which is in the form of:
\begin{equation*}
    \label{eq_mo}
m^{k+1} = \beta \nabla f(x^k; \xi^k) + (1-\beta) \left( m^k +\nabla f(x^k; \xi^k) - \nabla f(x^{k-1}; \xi^{k-1}) \right),
\end{equation*}
where $\nabla f(x^k; \xi^k)$ is the stochastic gradient estimator. 
This approach can also be viewed as a convex combination of the SGD and SARAH methods.

\subsection{Algorithm}
In the stochastic setting, problem \eqref{eq_problem} has the following form:
\begin{equation}
    \label{eq_sto_problem}
    \min_{x\in \mathbb{R}^d} f(x) = \frac{1}{n} \sum_{i=1}^n f_i (x) := \frac{1}{n} \sum_{i=1}^n \mathbb{E}_{\xi_i} F_i(x; \xi_i),
\end{equation}
where $\xi_i$ is a random variable and $f_i(x) = \mathbb{E}_{\xi_i} F_i(x; \xi_i)$.
We cannot directly obtain the value of the local function. 
But can only obtain an approximation through sampling.
Therefore, the zeroth-order gradient estimation involves two sources of randomness: one from the inherent randomness of the problem itself, and the other from the sampling of directions in zeroth-order gradient estimation, as mentioned in Section 2.1.
In the stochastic setting, the zeroth-order gradient estimation of local function $f_i$ is as follows:
\begin{equation}
\label{eq zero order gradient estimator}
    \hat{\nabla} f_i (x;u_i,\xi_i) := d \; \dfrac{F_i\left(x+t u_i; \xi_i\right)-F_i\left(x-t u_i; \xi_i\right)}{2 t}\; u_i.
\end{equation}

The detailed DZOVR is presented in Algorithm \ref{algo}.
It can be seen that each iteration requires 4 queries, except for the initial iteration of the algorithm.
In other words, the sampling complexity per iteration of the algorithm is roughly $\mathcal{O}(1)$, which is important in zeroth-order algorithms.

\begin{algorithm}[t]
    \label{algo}
        \caption{Decentralized Zeroth-Order Variance Reduced method (DZOVR)}
        \KwIn{$W$; $x_i^0, g_i^0 = m_i^{-1}= 0 \; (i=1,\dots,n)$; $b_0$; $\alpha > 0$; $0 \le \beta < 1$; positive sequence $\{t_k\}_{k=0}^{\infty}$.}
        \For{$i=1,\dots,n$ in parallel}{
            Sample $\{u^0_{i,s}, \xi^0_{i,s}\}_{s=1}^{b_0}$ independently;\\
            $\hat{\nabla} f_i (x_i^0;u_{i,s}^0,\xi_{i,s}^0)=d \; \dfrac{F_i\left(x_i^0+t_0 u_{i,s}^0; \xi_{i,s}^0\right)-F_i\left(x_i^0-t_0 u_{i,s}^0; \xi_{i,s}^0\right)}{2 t_0}\; u_{i,s}^0$;\\
            $m_i^0 = \frac{1}{b_0} \sum\limits_{s=1}^{b_0} \hat{\nabla} f_i\left(x_i^0 ; u_{i, s}^0, \xi_{i, s}^0\right)$;\\
            $g_i^{1}=\sum\limits_{j=1}^n w_{i j}\left(g_j^0+m_j^0-m_j^{-1}\right)$;\\
            $x_i^{1}=\sum\limits_{j=1}^n w_{i j}\left(x_j^0-\alpha g_j^{1}\right)$;\\
        }
        \For{$k=1,\dots$}{
            \For{$i=1,\dots,n$ in parallel}{
                Sample $\xi_i^k$ independently;\\
                Sample $u_i^k \sim \mathcal{U}(\mathbb{S}_{d-1})$ independently;\\
                $\hat{\nabla} f_i (x_i^k;u_i^k,\xi_i^k)=d \; \dfrac{F_i\left(x_i^k+t_k u_i^k; \xi_i^k\right)-F_i\left(x_i^k-t_k u_i^k; \xi_i^k\right)}{2 t_k}\; u_i^k$;\\
                $\hat{\nabla} f_i (x_i^{k-1};u_i^k,\xi_i^k)=d \; \dfrac{F_i\left(x_i^{k-1}+t_k u_i^k; \xi_i^k\right)-F_i\left(x_i^{k-1}-t_k u_i^k; \xi_i^k\right)}{2 t_k}\; u_i^k$;\\
                $m_i^k=\beta \hat{\nabla} f_i\left(x_i^k ; u_i^k, \xi_i^k\right)+(1-\beta)\left(m_i^{k-1}+\hat{\nabla} f_i\left(x_i^k ; u_i^k, \xi_i^k\right)-\hat{\nabla} f_i\left(x_i^{k-1} ; u_i^k, \xi_i^k\right)\right)$;\\
                $g_i^{k+1}=\sum\limits_{j=1}^n w_{i j}\left(g_j^k+m_j^k-m_j^{k-1}\right)$;\\
                $x_i^{k+1}=\sum\limits_{j=1}^n w_{i j}\left(x_j^k-\alpha g_j^{k+1}\right)$.
            }
        }
\end{algorithm}
Let $\{x_i^k, g_i^k, m_i^k, \hat{\nabla} f_i\left(x_i^k ; u_i^k, \xi_i^k\right)\}$ be the sequences generated by Algorithm \ref{algo}. For ease of notation, define
\begin{align}
\label{notation}
x^k=\begin{bmatrix}
    x_1^k \\
    x_2^k \\
    \vdots \\
    x_n^k\\
    \end{bmatrix}, ~g^k=\begin{bmatrix}
    g_1^k \\
    g_2^k \\
    \vdots \\
    g_n^k\\
    \end{bmatrix},~m^k=\begin{bmatrix}
    m_1^k \\
    m_2^k \\
    \vdots \\
    m_n^k\\
    \end{bmatrix},~\hat{\nabla} f\left(x^k ; u^k, \xi^k\right) = \begin{bmatrix}
    \hat{\nabla} f_1\left(x_1^k ; u_1^k, \xi_1^k\right) \\
    \hat{\nabla} f_2\left(x_2^k ; u_2^k, \xi_2^k\right) \\
    \vdots \\
    \hat{\nabla} f_n\left(x_n^k ; u_n^k, \xi_n^k\right) \\
    \end{bmatrix}\in\mathbb{R}^{nd}.
\end{align}
Recalling that $W$ is the consensus matrix, the update in Algorithm \ref{algo} can be rewritten as
\begin{align*}
m^k & =\beta \hat{\nabla} f\left(x^k ; u^k, \xi^k\right)+(1-\beta)\left(m^{k-1}+\hat{\nabla} f\left(x^k ; u^k, \xi^k\right)-\hat{\nabla} f\left(x^{k-1} ; u^k, \xi^k\right)\right), \\
g^{k+1} & =(W \otimes I_d)\left(g^k+m^k-m^{k-1}\right), \\
x^{k+1} & =(W \otimes I_d)\left(x^k-\alpha g^{k+1}\right).
\end{align*}

\subsection{Convergence Analysis}
Before establishing the convergence of DZOVR, we first introduce the following standard assumptions.
\begin{assumption}
\label{ass1}
    The consensus matrix $W$ is doubly stochastic and primitive, that is, $W1_n = 1_n$, $1_n^TW = 1_n^T$ and there exists a positive integer $k$ such that $W^k > 0$.
\end{assumption}

\begin{assumption}
\label{ass2}
    Each local function $f_i$ is $L$-smooth and $F_i\left(\cdot \; ;\xi_i\right)$ is L-smooth for almost all $\xi_i$.
\end{assumption}

\begin{assumption}
\label{ass3}
    For any $x\in \mathbb{R}^d$, the stochastic gradient $\nabla F_i\left(x ;\xi_i\right)$ satisfies
    \begin{align*}
        & \mathbb{E}_{\xi_i} \left[ \nabla F_i\left(x ;\xi_i\right) \right]  = \nabla f_i(x); \\
        & \mathbb{E}_{\xi_i} \left[ \Vert \nabla F_i\left(x ;\xi_i\right) - \nabla f_i(x) \Vert^2 \right] \le \sigma_0^2\left\|\nabla f_i(x)\right\|^2+\sigma_1^2.
    \end{align*}
\end{assumption}

\begin{assumption}
\label{ass4}
    For any $x\in \mathbb{R}^d$, there exists two constants $\sigma_2$ and $\sigma_3$ such that
    $$
    \left\|\nabla f_i(x)-\nabla f(x)\right\|^2 \leq \sigma_2^2\|\nabla f(x)\|^2+\sigma_3^2.
    $$
\end{assumption}

\begin{remark}
    Under Assumption \ref{ass1}, it can be shown that \cite{tang2020}
    $$
    \rho:=\left\| W -\frac{1_n 1_n^T}{n}\right\| \in[0,1).
    $$
    Assumption \ref{ass2}, as well as the first term in Assumption \ref{ass3}, are standard assumptions in the context of stochastic optimization \cite{xin2021,gao2023}. 
    It is worth noting that the smoothness of $F_i(\cdot,\xi)$ in Assumption \ref{ass2} is essential for variance reduction methods and standard for zeroth-order optimization \cite{yi2022,zhang2022}.
    The second term of Assumption \ref{ass3} is weaker than the standard assumption of bounded variance \cite{yi2022}. 
    When $\sigma_0 = 0$, it reduces to the assumption of bounded variance.
    Assumption \ref{ass4} concerns the dissimilarity property. 
    Compared to the standard bounded dissimilarity assumption, i.e., $
    \left\|\nabla f_i(x)-\nabla f(x)\right\|^2 \leq \tilde{\sigma}_3^2
    $, it is weaker and can be satisfied by functions such as quadratic functions \cite{yi2022}.
\end{remark}

We are now in the position to present our main results, whose proof is deferred to Appendix \ref{appendix_b}.
\begin{theorem}
    \label{thm}
Suppose Assumption 1-4 holds. Let the constant step size $\alpha$ and momentum parameter $\beta$ obey that
\small{
\begin{align*}
\alpha \le \min\left\{1, \frac{1}{2L},  \frac{1-\rho^2}{\sqrt{360d} L \rho^2}, \frac{\left(1-\rho^2\right)^2}{284 \sqrt{d} L \rho^2}, \frac{1}{2\sqrt{c_3}}, \frac{L d}{2\sqrt{n c_0 c_3}}, \frac{1}{4 \sqrt{ c_1 c_3}}, \sqrt{\frac{c_4}{4 c_3}}\right\},\quad 4c_2\leq \beta \leq \min \left\{1, \frac{L^2 d^2}{n c_0}, \frac{1}{4 c_1}, c_4\right\},
\end{align*}}where $\{c_i\} (i=0,\dots,4)$ are absolute constants given in Appendix \ref{appendix_b}. Provided that $t_0 \le \frac{\beta}{d^2}, \sum\limits_{k=0}^K t_k^2 \le \frac{\beta^2 M_t}{d^4}$, then one has
\begin{align*}
\frac{1}{K} \sum_{k=0}^{K-1} \mathbb{E}\left\|\nabla f\left(\bar{x}^k\right)\right\|^2 \leq & \frac{4\left(f\left(\bar{x}^0\right)-f^*\right)}{\alpha K}+ \frac{\left(192 d\left(1+\sigma_0^2\right)+48\right)\left\|\nabla f\left(x^0\right)\right\|^2}{n b_0 \beta K}+\frac{192 d \sigma_1^2}{b_0 \beta K}\\
& +\left(\frac{24 d^2 L^2}{b_0 \beta K}+\frac{48 L^2}{b_0 \beta K}+\frac{16 L^2}{\beta K}\right) t_0^2 +\frac{18560 \alpha^2 d L^2 \rho^2}{K \beta n\left(1-\rho^2\right)^3} e_g +\frac{335000 \alpha^2 d L^2 \rho^4}{K n\left(1-\rho^2\right)^4} e_m \\
& +\frac{1}{K}\left(208 L^2 + \frac{5800000 \alpha^2 \beta L^4 \rho^4}{\left(1-\rho^2\right)^4}\right) M_t \\
& +\left(\frac{256 \beta d}{n}+\frac{1044000 \alpha^2 \beta d^2 L^2 \rho^4}{\left(1-\rho^2\right)^3}+\frac{12064000 \alpha^2 \beta^2 d^2 L^2 \rho^4}{\left(1-\rho^2\right)^4}\right)\left(1+\sigma_0^2\right) \sigma_3^2 \\
& +\left(\frac{64 \beta d}{n}+\frac{232000 \alpha^2 \beta d^2 L^2 \rho^4}{\left(1-\rho^2\right)^3}+\frac{2726000 \alpha^2 \beta^2 d^2 L^2 \rho^4}{\left(1-\rho^2\right)^4}\right) \sigma_1^2,
\end{align*}
where
\begin{align*}
& \bar{x}^k = \frac{1}{n}\sum_{i=1}^n x_i^k,\\
& e_m = \frac{24 d\left(1+\sigma_0^2\right)+6}{b_0}\left\|\nabla f\left(x^0\right)\right\|^2+\frac{24 n d \sigma_1^2}{b_0} +\frac{3 n d^2 L^2 }{b_0} t_0^2+\frac{6 n L^2}{b_0} t_0^2 +2 n L^2 t_0^2,\\
& e_g =2 \rho^2 e_m+2 \rho^2\left\|\nabla f\left(x^0\right)\right\|^2.
\end{align*}
\end{theorem}

\begin{remark}
    Let $\tilde{\sigma}^2 := (1+\sigma_0^2)\sigma_3^2 + \sigma_1^2$. If $b_0 = \mathcal{O}(d)$, the expected mean-squared stationary gap $\frac{1}{K} \sum_{k=0}^{K-1} \mathbb{E}\left\|\nabla f\left(\bar{x}^k\right)\right\|^2$ of DZOVR will decay with the rate of $\mathcal{O}\left(\frac{d}{K}\right)$ up to a steady-state error. 
    If we further let $\beta = \frac{\alpha^2}{n}$, the steady-state error is $\mathcal{O}\left(\frac{\alpha^2 d \tilde{\sigma}^2}{n^2}\right) + \mathcal{O}\left(\alpha^4\right)$, which is dominated by $\mathcal{O}\left(\frac{\alpha^2 d \tilde{\sigma}^2}{n^2}\right)$ when $\alpha$ is small enough. 
    Compared to methods without variance reduction, whose steady-state error is $\mathcal{O}\left(\frac{\alpha d \tilde{\sigma}^2}{n}\right)$, the variance is reduced by a factor of $\frac{\alpha}{n}$. 
\end{remark}

Let $a_0 = \min\left\{1, \frac{1}{2L},  \frac{1-\rho^2}{\sqrt{360d} L \rho^2}, \frac{\left(1-\rho^2\right)^2}{284 \sqrt{d} L \rho^2}, \frac{1}{2\sqrt{c_3}}, \frac{L d}{2\sqrt{n c_0 c_3}}, \frac{1}{4 \sqrt{ c_1 c_3}}, \sqrt{\frac{c_4}{4 c_3}}\right\}$. We have the following corollary.
\begin{corollary}
\label{cor}
Suppose Assumption 1-4 holds. Let $\alpha = \frac{1}{100L}\cdot\frac{n^{2/3}}{d^{2/3} K^{1/3}}$, $\beta=\frac{n^{1/3}}{d^{1/3} K^{2/3}}$, $b_0 = \left\lceil d^{2 / 3} (nK)^{1 / 3} \right\rceil$ and $t_k = \frac{\beta}{d^2 (k+1)^{1/4}}$ in Theorem \ref{thm}.
Then for $K \ge \frac{(100L)^3 n^2}{d^2 a_0^3}$, we have
    $$
    \frac{1}{K} \sum_{k=0}^{K-1} \mathbb{E}\left\|\nabla f\left(\bar{x}^k\right)\right\|^2 = \mathcal{O}\left(\left(\dfrac{d}{nK}\right)^{2/3}\right).
    $$
\end{corollary}

\begin{remark}
    Corollary \ref{cor} implies that for a sufficiently large $K$, DZOVR achieves a convergence rate of $\mathcal{O}((d/nK)^{2/3})$, which is better than the other existing zeroth-order algorithms \cite{yi2022,zhang2021,zhang2022}.
    In other words, the best sampling complexity to reach an $\epsilon$-accurate stationary point achieved by existing distributed zeroth-order algorithms is $\mathcal{O}(dn^{-1}\epsilon^{-4})$ while our algorithm improves it to $\mathcal{O}(dn^{-1}\epsilon^{-3})$.
    Furthermore, Corollary \ref{cor} also demonstrates that DZOVR has the network-independent and linear speedup properties\footnote{ 
    Linear speedup means that the number of stochastic gradient computations required at each node in the network is reduced by a factor of $1/n$.}.
\end{remark}

\begin{remark}
The best $\mathcal{O}(n^{-1}\epsilon^{-3})$ sampling complexity for the distributed first-order methods is established in \cite{xin2021hybrid}.
It can be observed that the sampling complexity of DZOVR only looses a factor of $d$ compared with that first-order method.
It's worth noting that in zeroth-order algorithms, gradient estimates are biased, and gradient tracking does not achieve consensus, which poses new challenges in the analysis.
\end{remark}

\section{Numerical Experiments}
\label{sec3}

In this section, we compare DZOVR with other state-of-the-art algorithms and validate the theoretical results through numerical experiments. 
To this end, we follow the problem setup in \cite{zhang2022,liu2018svrg,zhang2021}, where 
$$
f(x)= \dfrac{1}{n} \sum_{i=1}^n \left( \dfrac{1}{m} \sum_{j=1}^m \Big( y_{ij}-\frac{1}{1+e^{-x^{\top} a_{ij}}} \Big)^2 \right).
$$
Let $d = 20$, $n = 20$ and $m=100$.
Given a reference $x = 1_d $, we sample $a_{ij}$ from $\mathcal{N}(0,I)$. If $1/(1+e^{-x^{\top} a_{ij}}) \ge 0.5$, then set $y_{ij} = 1$; otherwise, set $y_{ij} = 0 $. 

\begin{figure}[ht]
    \centering
    \subfigure{
        \includegraphics[width=2.5in]{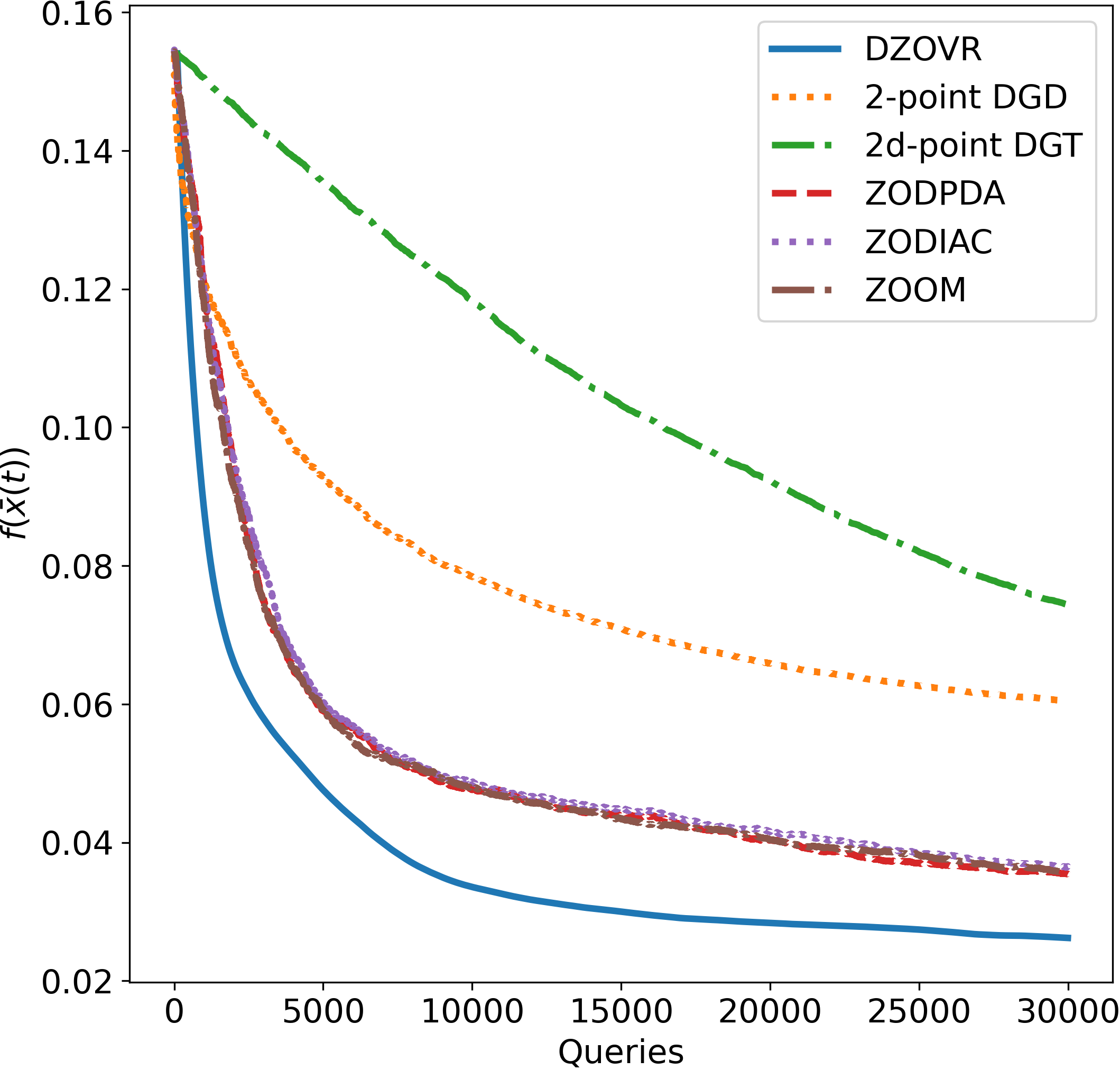}
    }
    \subfigure{
    	\includegraphics[width=2.5in]{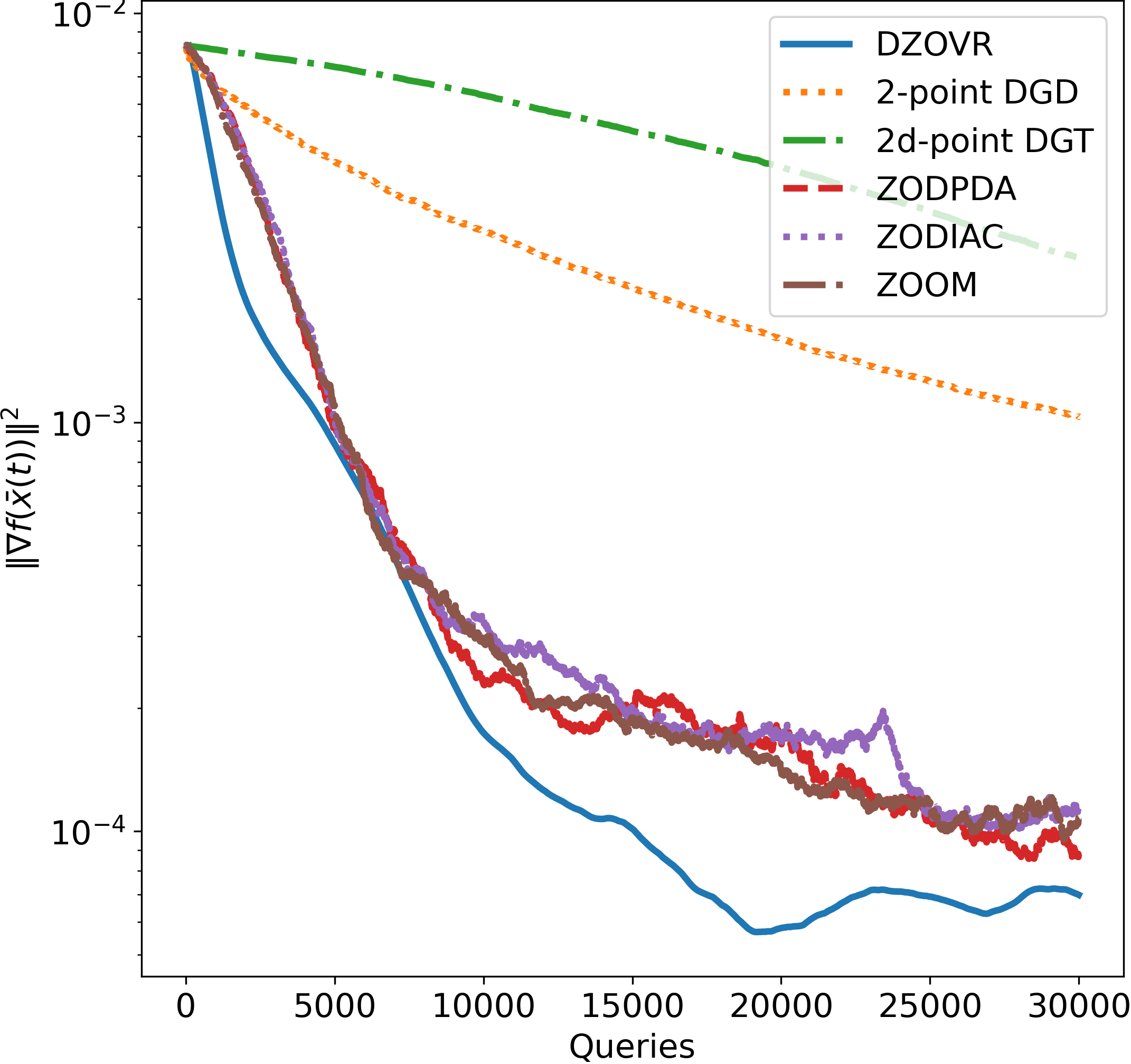}
    }
    \caption{Comparison of the performance of different algorithms.}
    \label{fig1}
\end{figure}

We first compare DZOVR with the following decentralized zeroth-order algorithms: 2-point DGD \cite{tang2020}, 2d-point DGT \cite{tang2020}, ZODPDA \cite{yi2022}, ZODIAC \cite{zhang2021}, and ZOOM \cite{zhang2022}. 
For 2-point DGD and 2d-point DGT, their step sizes are set as $0.1/k$ and 0.02 respectively.
For the other three methods, the step size is set to 0.01.
In DZOVR, we set $\beta=0.001$, $\alpha=0.05$ and $b_0 = 100$.
Here the graph is created by first sampling $n$ points on $\mathbb{S}^2$ uniformly at random, and then linking pairs of points with spherical distances smaller than $\pi/4$ \cite{tang2020}.
Metropolis-Hastings weights \cite{xiao2005} are used to build $W$.
The results of the numerical experiments are displayed in Figure \ref{fig1}. 
It's worth noting that the horizontal axis in the figure represents queries because in zeroth-order algorithms, the focus is on achieving better results with as few queries as possible, rather than the number of iterations.
Though DZOVR requires 4 queries in each iteration while the other methods require only 2 queries (except for 2d-point DGT), it still outperforms the other algorithms due to its faster convergence.

\begin{figure}[ht]
    \centering
    \subfigure{
        \includegraphics[width=2.5in]{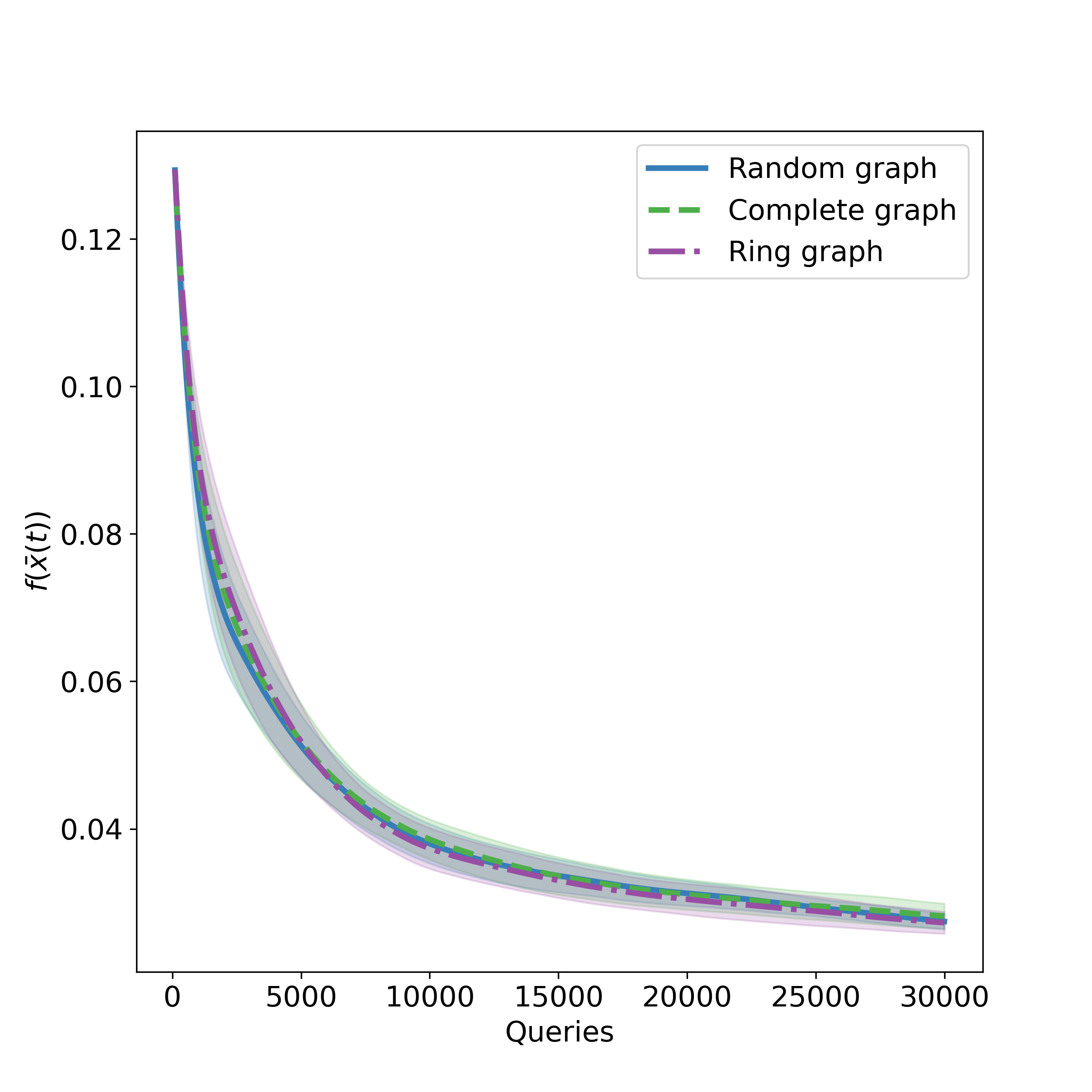}
    }
    \subfigure{
    	\includegraphics[width=2.5in]{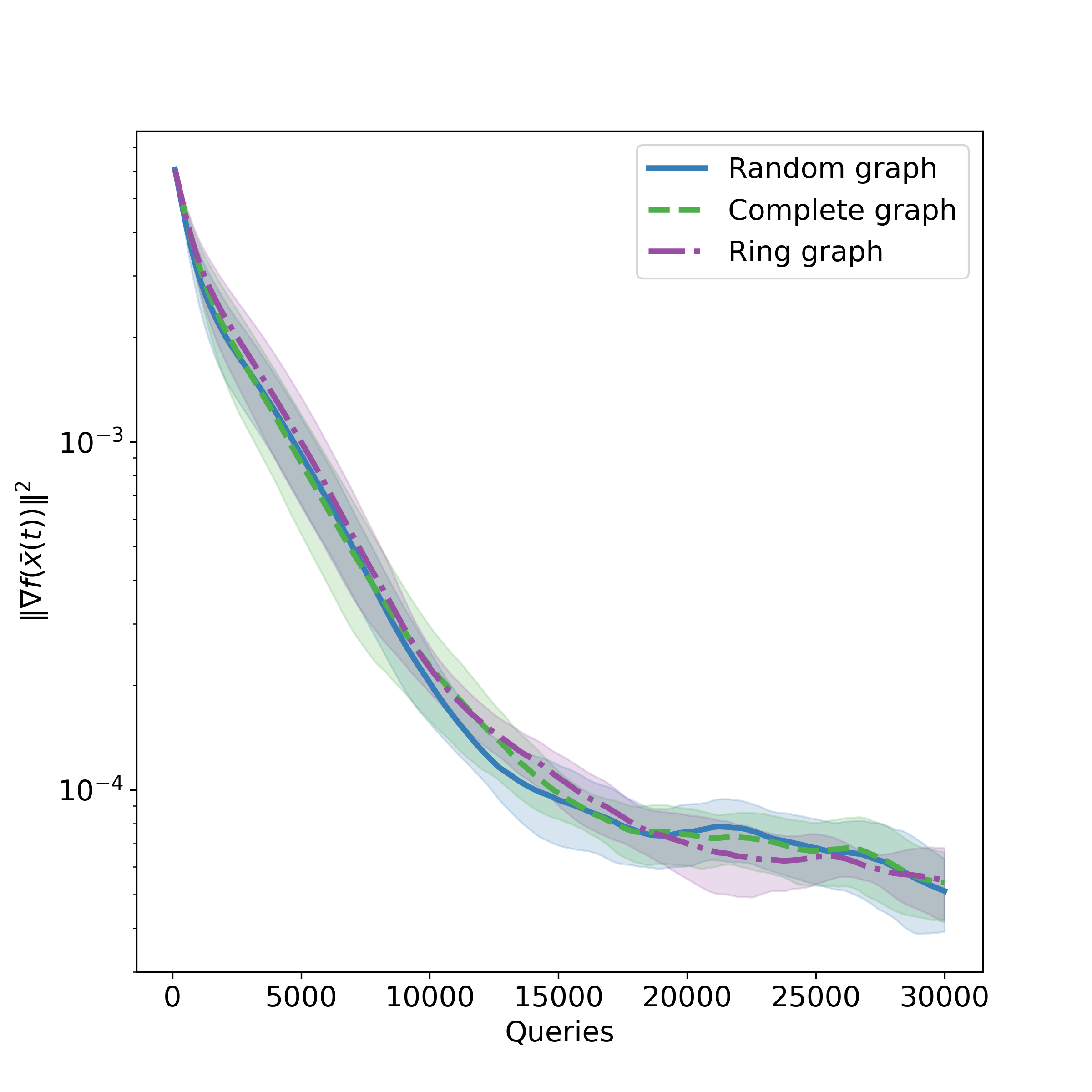}
    }
    \caption{Performance of DZOVR under three different network topologies. The curves in the graph represent the average results out of 10 random trials, and the shaded areas represent the standard deviation.}
    \label{fig2}
\end{figure}

To investigate the network-independent property of DZOVR, we conduct numerical experiments on three different network structures: the random graph mentioned above, the complete graph, and the ring graph.
The mean objective function and gradient values as well as the standard deviations out of 10 repeated random tests are presented in Figure \ref{fig2}. 
It can be observed that DZOVR performs consistently well across the three network structures, indicating its network-independent property.

\begin{figure}[ht]
    \centering
    \includegraphics[width=2.5in]{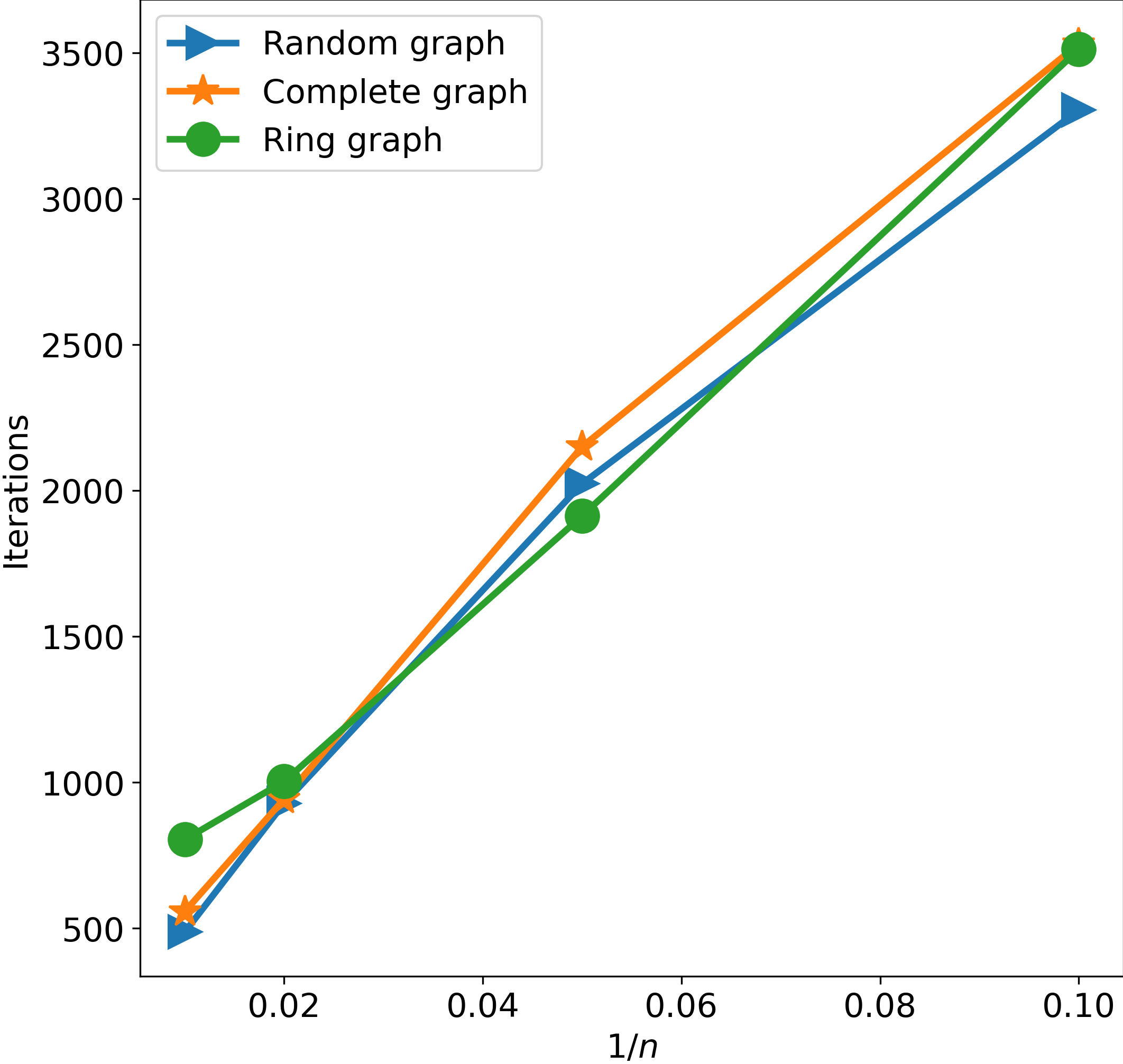}
    \caption{The linear speedup performance of DZOVR. The value of $n$ are $\{10, 20, 50, 100\}$.}
    \label{fig3}
\end{figure}
Regarding the linear speedup property of the algorithm, experiments are conducted with $n = \{10,20,50,100\}$, $\alpha = 0.005n^{2/3}$ on the three aforementioned graphs. 
We record the number of iterations required for the squared norm of the gradient to be less than 0.001.
The results are shown in Figure \ref{fig3}, where the $x$-axis represents $1/n$ and the $y$-axis represents the number of iterations required to achieve the specified accuracy.
A desirable linear speedup can be observed from the figure.

\section{Conclusion}
\label{sec4}
We propose a decentralized zeroth-order variance reduced algorithm to mitigate the adverse effects of excessive variance in distributed zeroth-order optimization. 
It is proved that the algorithm convergences at a rate of $\mathcal{O}((d/{nK})^{2/3})$. 
Numerical experiments demonstrate its superior performance over state-of-the-art algorithms and its network-independent and linear speedup properties.
Interesting directions for future works include analyzing the algorithm under the PL condition, accelerating the algorithm, or extending it to the non-smooth setting.

\bibliographystyle{plain}
\bibliography{sample}
\appendix

\section{Useful Lemmas}
We first present several useful lemmas that will be utilized in the proof of Theorem \ref{thm}.

\begin{lemma}[Basic properties of 2-point estimator \cite{tang2020}] \quad
\label{lemma1}
    \begin{itemize}
        \item[1.] 
            For a $L$-smooth function $f$,
            $$
            \left|\frac{f(x+t u)-f(x-t u)}{2 t}-\langle\nabla f(x), u\rangle\right| \leq \frac{1}{2} t L\|u\|^2
            $$
            holds for any $x, u \in \mathbb{R}^d$.
        \item[2.] 
            For the 2-point estimator $\hat{\nabla} f(x; u)= d \; \dfrac{f(x+t u)-f(x-t u)}{2 t} \cdot u$, we have
            $$
            \mathbb{E}_{u \sim \mathcal{U}\left(\mathbb{S}_{d-1}\right)}[\hat{\nabla} f(x; u)]=\nabla f_t(x),
            $$
            where $f_t(x):=\mathbb{E}_{y \sim \mathcal{U}\left(\mathbb{B}_d\right)}[f(x+t y)]$. In addition, if $f$ is $L$-smooth, then $f_t$ is also $L$-smooth, and the following inequality holds:
            $$
            \left\|\nabla f(x)-\nabla f_t(x)\right\| \leq t L.
            $$
        \item[3.] 
            For any deterministic $h \in \mathbb{R}^d$, we have
            $$
            \mathbb{E}_{z \sim \mathcal{U}\left(\mathbb{S}_{d-1}\right)}[d \cdot\langle h, z\rangle z]=h \quad \text{and} \quad \mathbb{E}_{z \sim \mathcal{U}\left(\mathbb{S}_{d-1}\right)}\left[d \cdot\langle h, z\rangle^2\right]=\|h\|^2.
            $$
    \end{itemize}
\end{lemma}
For the sake of clarity, we define
\begin{align*}
\bar{m}^k &= \frac{1}{n} \sum_{i=1}^n m_i^k,\\
\nabla \bar{f}^k&=\frac{1}{n} \sum_{i=1}^n \nabla f_{i}\left(x_i^k\right),\\
\hat{\nabla} \bar{f}\left(x^k ; u^k, \xi^k\right) &= \frac{1}{n} \sum_{i=1}^n \hat{\nabla} f_i\left(x_i^k ; u_i^k, \xi_i^k\right),\\
\nabla f_{i, t_k}\left(x_i^k\right)&=\mathbb{E}_{u \sim \mathcal{U}\left(\mathbb{S}_{d-1}\right)}\left[\hat{\nabla} f_i\left(x_i^k\right)\right],\\
\nabla f^k &= \begin{bmatrix}
    \nabla f_1(x_i^k)^\top &\cdots &\nabla f_n(x_n^k)^\top
\end{bmatrix}^\top\in\mathbb{R}^{nd}.
\end{align*}

\begin{lemma}
\label{lemma2}
Suppose $\{x^k, u^k, \xi^k\}$ is a sequence generated by Algorithm \ref{algo}. Under Assumption \ref{ass2}, \ref{ass3}, and \ref{ass4}, for any $k\geq 1$, one has
    \begin{align*}
    \mathbb{E}\left[\left\|\hat{\nabla} \bar{f}\left(x^k ; u^k, \xi^k\right)\right\|^2 \mid \mathcal{F}_k\right] \leq & \frac{1}{2} L^2 d^2 t_k^2+\left(\frac{4 d L^2\left(1+\sigma_0^2\right)}{n^2}+\frac{4 L^2}{n}\right)\left\|x^k-1_n \otimes \bar{x}^k\right\|^2 \\
    & +\left(\frac{8 d}{n}\left(1+\sigma_0^2\right)\left(1+\sigma_2^2\right)+4\right)\left\|\nabla f\left(\bar{x}^k\right)\right\|^2+\frac{8 d}{n}\left(1+\sigma_0^2\right) \sigma_3^2+\frac{2 d}{n} \sigma_1^2.
    \end{align*}
\end{lemma}

\begin{proof}
Conditioned on $\xi^k$ and $\mathcal{F}_k$, a straightforward computation yields that
    \begin{align*}
    & \mathbb{E}_{u^k}\left\|\hat{\nabla} \bar{f}\left(x^k ; u^k, \xi^k\right)\right\|^2 \\
    = & \mathbb{E}_{u^k}\left\|\frac{1}{n} \sum_{i=1}^n d \frac{F_i\left(x_i^k+t_k u_i^k ; \xi_i^k\right)-F_i\left(x_i^k-t_k u_i^k ; \xi_i^k\right)}{2 t_k} u_i^k\right\|^2 \\
    = & \mathbb{E}_{u^k}\left\|\frac{d}{n} \sum_{i=1}^n \frac{F_i\left(x_i^k+t_k u_i^k ; \xi_i^k\right)-F_i\left(x_i^k-t_k u_i^k ; \xi_i^k\right)}{2 t_k} u_i^k-\left\langle\nabla F_i\left(x_i^k ; \xi_i^k\right), u_i^k\right\rangle u_i^k+\left\langle\nabla F_i\left(x_i^k ; \xi_i^k\right), u_i^k\right\rangle u_i^k\right\|^2 \\
    \leq & \frac{2 d^2}{n} \sum_{i=1}^n \mathbb{E}_{u^k}\left\|\frac{F_i\left(x_i^k+t_k u_i^k ; \xi_i^k\right)-F_i\left(x_i^k-t_k u_i^k ; \xi_i^k\right)}{2 t_k} u_i^k-\left\langle\nabla F_i\left(x_i^k ; \xi_i^k\right), u_i^k\right\rangle u_i^k\right\|^2 \\
    & +2 \mathbb{E}_{u^k}\left\|\frac{d}{n} \sum_{i=1}^n\left\langle\nabla F_i\left(x_i^k ; \xi_i^k\right), u_i^k\right\rangle u_i^k\right\|^2 \\
    \leq & \frac{2 d^2}{n} \cdot n \cdot \frac{1}{4} t_k^2 L^2+\frac{2(d-1)}{n^2} \sum_{i=1}^n\left\|\nabla F_i\left(x_i^k ; \xi_i^k\right)\right\|^2+2\left\|\frac{1}{n} \sum_{i=1}^n \nabla F_i\left(x_i^k ; \xi_i^k\right)\right\|^2, \numberthis \label{eq tmp1}
    \end{align*}
    where the last inequality is due to Lemma \ref{lemma1} and the fact that 
    \begin{align*}
        \mathbb{E}_{u^k}\left\|\frac{d}{n} \sum_{i=1}^n\left\langle\nabla F_i\left(x_i^k ; \xi_i^k\right), u_i^k\right\rangle u_i^k\right\|^2 \leq \frac{(d-1)}{n^2} \sum_{i=1}^n\left\|\nabla F_i\left(x_i^k ; \xi_i^k\right)\right\|^2+\left\|\frac{1}{n} \sum_{i=1}^n \nabla F_i\left(x_i^k ; \xi_i^k\right)\right\|^2.
    \end{align*}
    The fact can be proved as follows:
    \begin{align*}
    & \mathbb{E}_{u^k}\left\|\frac{d}{n} \sum_{i=1}^n\left\langle\nabla F_i\left(x_i^k ; \xi_i^k\right), u_i^k\right\rangle u_i^k\right\|^2 \\
    = & \frac{d^2}{n^2} \mathbb{E}_{u^k}\left[\sum_{i=1}^n\left|\left\langle\nabla F_i\left(x_i^k ; \xi_i^k\right), u_i^k\right\rangle\right|^2+\sum_{i \neq j}\left\langle\nabla F_i\left(x_i^k ; \xi_i^k\right), u_i^k\right\rangle\left\langle\nabla F_j\left(x_j^k ; \xi_j^k\right), u_j^k\right\rangle \cdot\left\langle u_i^k, u_j^k\right\rangle\right] \\
    = & \frac{d}{n^2} \sum_{i=1}^n\left\|\nabla F_i\left(x_i^k ; \xi_i^k\right)\right\|^2+\frac{1}{n^2} \sum_{i \neq j}\left\langle\nabla F_i\left(x_i^k ; \xi_i^k\right), \nabla F_j\left(x_j^k ; \xi_j^k\right)\right\rangle \\
    = & \frac{(d-1)}{n^2} \sum_{i=1}^n\left\|\nabla F_i\left(x_i^k ; \xi_i^k\right)\right\|^2+\frac{1}{n^2}\left\|\sum_{i=1}^n \nabla F_i\left(x_i^k ; \xi_i^k\right)\right\|^2 \\
    = & \frac{(d-1)}{n^2} \sum_{i=1}^n\left\|\nabla F_i\left(x_i^k ; \xi_i^k\right)\right\|^2+\left\|\frac{1}{n} \sum_{i=1}^n \nabla F_i\left(x_i^k ; \xi_i^k\right)\right\|^2,
    \end{align*}
    where the second line is due to Lemma \ref{lemma1}.

    Thus, we have
    \begin{align*}
    & \mathbb{E}\left[\left\|\hat{\nabla} \bar{f}\left(x^k ; u^k, \xi^k\right)\right\|^2 \mid \mathcal{F}_k\right]\\
    = & \mathbb{E}_{\xi^k} \mathbb{E}_{u^k}\left\|\hat{\nabla} \bar{f}\left(x^k ; u^k, \xi^k\right)\right\|^2 \\
    \leq & \frac{1}{2} L^2 d^2 t_k^2+\frac{2(d-1)}{n^2} \sum_{i=1}^n \mathbb{E}_{\xi^k}\left\|\nabla F_i\left(x_i^k, \xi_i^k\right)\right\|^2+2 \mathbb{E}_{\xi^k}\left\|\frac{1}{n} \sum_{i=1}^n \nabla F_i\left(x_i^k, \xi_i^k\right)\right\|^2 \\
    \stackrel{(a)}{=}& \frac{1}{2} L^2 d^2 t_k^2+\frac{2(d-1)}{n^2} \sum_{i=1}^n \mathbb{E}_{\xi^k}\left\|\nabla F_i\left(x_i^k, \xi_i^k\right)-\nabla f_i\left(x_i^k\right)\right\|^2+\frac{2(d-1)}{n^2} \sum_{i=1}^n\left\|\nabla f_i\left(x_i^k\right)\right\|^2 \\
    & +2 \mathbb{E}_{\xi^k} \| \frac{1}{n} \sum_{i=1}^n \nabla F_i\left(x_i^k; \xi_i^k\right)-\frac{1}{n} \sum_{i=1}^n \nabla f_i\left(x_i^k\right)\left\|^2+2\right\| \frac{1}{n} \sum_{i=1}^n \nabla f_i\left(x_i^k\right) \|^2 \\
    \stackrel{(b)}{\leq} & \frac{1}{2} L^2 d^2 t_k^2+\frac{2(d-1)}{n^2} \sum_{i=1}^n\left(\sigma_0^2\left\|\nabla f_i\left(x_i^k\right)\right\|^2+\sigma_1^2\right)+\frac{2(d-1)}{n^2} \sum_{i=1}^n\left\|\nabla f_i\left(x_i^k\right)\right\|^2 \\
    & +\frac{2}{n^2} \sigma_0^2 \sum_{i=1}^n\left\|\nabla f_i\left(x_i^k\right)\right\|^2+\frac{2}{n} \sigma_1^2+2\left\|\frac{1}{n} \sum_{i=1}^n \nabla f_i\left(x_i^k\right)\right\|^2 \\
    \leq & \frac{1}{2} L^2 d^2 t_k^2+\frac{2 d}{n^2}\left(1+\sigma_0^2\right)\left\|\nabla f^k\right\|^2+\frac{2 d}{n} \sigma_1^2+2\left\|\frac{1}{n} \sum_{i=1}^n \nabla f_i\left(x_i^k\right)\right\|^2 \\
    \stackrel{(c)}{\leq} & \frac{1}{2} L^2 d^2 t_k^2+\frac{2 d}{n^2}\left(1+\sigma_0^2\right)\left(2 L^2\left\|x^k-1_n \otimes \bar{x}^k\right\|^2+4 n\left(1+\sigma_2^2\right)\left\|\nabla f\left(\bar{x}^k\right)\right\|^2+4 n \sigma_3^2\right) \\
    & +4\left\|\nabla f\left(\bar{x}^k\right)\right\|^2+\frac{4 L^2}{n}\left\|x^k-1_n \otimes \bar{x}^k\right\|^2+\frac{2 d}{n} \sigma_1^2 \\
    = & \frac{1}{2} L^2 d^2 t_k^2+\left(\frac{4 d L^2\left(1+\sigma_0^2\right)}{n^2}+\frac{4 L^2}{n}\right)\left\|x^k-1_n \otimes \bar{x}^k\right\|^2 \\
    & +\left(\frac{8 d}{n}\left(1+\sigma_0^2\right)\left(1+\sigma_2^2\right)+4\right)\left\|\nabla f\left(\bar{x}^k\right)\right\|^2+\frac{8 d}{n}\left(1+\sigma_0^2\right) \sigma_3^2+\frac{2 d}{n} \sigma_1^2, 
    \end{align*}
    where step (a) and step (b) follow from Assumption \ref{ass3}, step (c) uses the fact
    \begin{align*}
    \|\nabla f^k\|_2^2 &\leq 2L^2\|x^k - 1_n\otimes \bar{x}^k\|^2 + 4n(1+\sigma_2^2)\|\nabla f(\bar{x}^k)\|^2 + 4n\sigma_3^2,\\
    \left\|\frac{1}{n}\sum_{i=1}^n \nabla f_i(x_i^k)\right\|^2&\leq \frac{2L^2}{n}\|x^k - 1_n\otimes \bar{x}^k\|^2   + 2\left\| \nabla f(\bar{x}^k)\right\|^2.\numberthis \label{tmp5}
    \end{align*}
Moreover, the fact used in step (c) can be proved as follows:
\begin{align*}
    \|\nabla f^k\|^2 &= \sum_{i=1}^n \|\nabla f_i(x_i^k)\|^2\\
    &\leq 2\sum_{i=1}^n \|\nabla f_i(x_i^k) - \nabla f_i(\bar{x}^k)\|^2 + 2\sum_{i=1}^n \|\nabla f_i(\bar{x}^k)\|^2 \\
    &\leq 2\sum_{i=1}^n \|\nabla f_i(x_i^k) - \nabla f_i(\bar{x}^k)\|^2 + 4\sum_{i=1}^n \|\nabla f_i(\bar{x}^k) - \nabla f(\bar{x}^k)\|^2 + 4\sum_{i=1}^n \|\nabla f(\bar{x}^k)\|^2\\
    &\stackrel{(a)}{\leq} 2L^2\sum_{i=1}^n \|x_i^k - \bar{x}^k\|^2 + 4\sum_{i=1}^n\left(\sigma_2^2 \|\nabla f(\bar{x})\|^2 + \sigma_3^2\right) + 4\sum_{i=1}^n \|\nabla f(\bar{x}^k)\|^2\\
    &=2L^2\|x^k - 1_n\otimes \bar{x}^k\|^2 + 4n(1+\sigma_2^2)\|\nabla f(\bar{x}^k)\|^2 + 4n\sigma_3^2,
\end{align*}
where step (a) follows from Assumption \ref{ass2} and Assumption \ref{ass4}. Using the same argument as above, one can obtain
\begin{align*}
    \left\|\frac{1}{n}\sum_{i=1}^n \nabla f_i(x_i^k)\right\|^2 &\leq 2\left\|\frac{1}{n}\sum_{i=1}^n \nabla f_i(x_i^k) - \nabla f(\bar{x}^k)\right\|^2 + 2\|\nabla f(\bar{x}^k)\|^2\\
    &=2\left\|\frac{1}{n}\sum_{i=1}^n \left(\nabla f_i(x_i^k) - \nabla f_i(\bar{x}^k) \right)\right\|^2 + 2\|\nabla f(\bar{x}^k)\|^2\\
    &\stackrel{(a)}{\leq}\frac{2}{n}\sum_{i=1}^n\| \nabla f_i(x_i^k) - \nabla f_i(\bar{x}^k) \|^2 +2\left\| \nabla f(\bar{x}^k)\right\|^2\\
    &\leq \frac{2L^2}{n}\sum_{i=1}^n\|x_i^k - \bar{x}^k\|^2  + 2\left\| \nabla f(\bar{x}^k)\right\|^2\\
    &=\frac{2L^2}{n}\|x^k - 1_n\otimes \bar{x}^k\|^2   + 2\left\| \nabla f(\bar{x}^k)\right\|^2,
\end{align*}
where step (a) is due to the Jensen's inequality. \end{proof}

\begin{lemma}
\label{lemma3}
Consider the same setup as stated in Lemma \ref{lemma2}. 
One has
    \begin{align*}
\mathbb{E}\left[\left\|\hat{\nabla} f\left(x^k ; u^k, \xi^k\right)\right\|^2 \mid \mathcal{F}_k\right] \leq &\frac{1}{2} n d^2 L^2 t_k^2+6 d L^2\left(1+\sigma_0^2\right)\left\|x^k-1_n \otimes \bar{x}^k\right\|^2\\
    &\quad +6 n d\left(1+\sigma_0^2\right)\left(1+\sigma_2^2\right)\left\|\nabla f\left(\bar{x}^k\right)\right\|^2+6 n d\left(1+\sigma_0^2\right) \sigma_3^2+2 n d \sigma_1^2.
    \end{align*}
\end{lemma}

\begin{proof}
Recall the definition of $\hat{\nabla} f\left(x^k ; u^k, \xi^k\right)$ in \eqref{notation}. A direct computation yields that
    \begin{align*}
    & \mathbb{E}\left[\left\|\hat{\nabla} f\left(x^k ; u^k, \xi^k\right)\right\|^2 \mid \mathcal{F}_k\right] \\
    = & \mathbb{E}_{\xi^k} \mathbb{E}_{u^k}\left[\sum_{i=1}^n\left\|\hat{\nabla} f_i\left(x_i^k ; u_i^k, \xi_i^k\right)\right\|^2\right] \\
    = & \mathbb{E}_{\xi^k} \mathbb{E}_{u^k} \sum_{i=1}^n\left\|d \frac{F_i\left(x_i^k+t_k u_i^k ; \xi_i^k\right)-F_i\left(x_i^k-t_k u_i^k ; \xi_i^k\right)}{2 t_k} u_i^k\right\|^2 \\
    = & \mathbb{E}_{\xi^k} \mathbb{E}_{u^k} \sum_{i=1}^n\left\|d \frac{F_i\left(x_i^k+t_k u_i^k ; \xi_i^k\right)-F_i\left(x_i^k-t_k u_i^k ; \xi_i^k\right)}{2 t_k} u_i^k-d\left\langle\nabla F_i\left(x_i^k ; \xi_i^k\right), u_i^k \right\rangle u_i^k+d \left\langle \nabla F_i\left(x_i^k ; \xi_i^k\right), u_i^k\right\rangle u_i^k\right\|^2 \\
    \leq & 2d^2 \mathbb{E}_{\xi^k} \mathbb{E}_{u^k} \sum_{i=1}^n \left\|\frac{F_i\left(x_i^k+t_k u_i^k ; \xi_i^k\right)-F_i\left(x_i^k-t_k u_i^k ; \xi_i^k\right)}{2 t_k} u_i^k-\left\langle \nabla F_i\left(x_i^k ; \xi_i^k\right), u_i^k \right\rangle u_i^k\right\|^2 \\
    & +2 \mathbb{E}_{\xi^k} \mathbb{E}_{u^k} \sum_{i=1}^n \| d\left\langle\nabla F_i\left(x_i^k ; \xi_i^k\right), u_i^k \right\rangle u_i^k \|^2 \\
    \stackrel{(a)}{\leq} & 2d^2\cdot \frac{1}{4}nL^2t_k^2 +2 d \sum_{i=1}^n \mathbb{E}_{\xi^k}\left\|\nabla F_i\left(x_i^k ; \xi_i^k\right)\right\|^2 \\
    \stackrel{(b)}{=} & \frac{1}{2} n d^2 L^2 t_k^2+2 d \sum_{i=1}^n\left\|\nabla f_i\left(x_i^k\right)\right\|^2+2 d \sum_{i=1}^n\left\|\nabla F_i\left(x_i^k ; \xi_i^k\right)-\nabla f_i\left(x_i^k\right)\right\|^2 \\
    \stackrel{(c)}{\leq} &\frac{1}{2} n d^2 L^2 t_k^2+2 d \sum_{i=1}^n\left\|\nabla f_i\left(x_i^k\right)\right\|^2+2 d \sum_{i=1}^n\left( \sigma_0^2 \|\nabla f_i(x_i^k)\|^2 + \sigma_1^2\right)\\
    =& \frac{1}{2} n d^2 L^2 t_k^2+2 d\left(1+\sigma_0^2\right) \sum_{i=1}^n\left\|\nabla f_i\left(x_i^k\right)\right\|^2+2 n d \sigma_1^2, \numberthis \label{eq tmp2}
    \end{align*}
    where step (a) is due to Lemma \ref{lemma1}, step (b) uses the fact that $\mathbb{E}_{\xi_i^k}\left[\langle \nabla F_i(x_i^k; \xi_i^k) - \nabla f_i(x_i^k), \nabla f_i(x_i^k)\rangle\right] = 0$, and step (c) follows from Assumption \ref{ass3}.
    Moreover, one can show that
    \begin{align*}
    \sum_{i=1}^n\left\|\nabla f_i\left(x_i^k\right)\right\|^2 & =\sum_{i=1}^n\left\|\nabla f_i\left(x_i^k\right)-\nabla f_i\left(\bar{x}^k\right)+\nabla f_i\left(\bar{x}^k\right)-\nabla f\left(\bar{x}^k\right)+\nabla f\left(\bar{x}^k\right)\right\|^2 \\
    & \leq 3 \sum_{i=1}^n\left\|\nabla f_i\left(x_i^k\right)-\nabla f_i\left(\bar{x}^k\right)\right\|^2+3 \sum_{i=1}^n\left\|\nabla f_i\left(\bar{x}^k\right)-\nabla f\left(\bar{x}^k\right)\right\|^2+3 n\left\|\nabla f\left(\bar{x}^k\right)\right\|^2 \\
    & \leq 3 L^2\left\|x^k-1_n \otimes \bar{x}^k\right\|^2+3 \sum_{i=1}^n\left(\sigma_2^2\left\|\nabla f\left(\bar{x}^k\right)\right\|^2+\sigma_3^2\right)+3 n\left\|\nabla f\left(\bar{x}^k\right)\right\|^2 \\
    & =3 L^2\left\|x^k-1_n \otimes \bar{x}^k\right\|^2+3 n\left(1+\sigma_2^2\right)\left\|\nabla f\left(\bar{x}^k\right)\right\|^2+3 n \sigma_3^2,\numberthis\label{tmp3}
    \end{align*}
    where the third line is due to Assumption \ref{ass2} and Assumption \ref{ass4}. Plugging \eqref{tmp3} into \eqref{eq tmp2}, we get
    \begin{align*}
\mathbb{E}\left[\left\|\hat{\nabla} f\left(x^k ; u^k, \xi^k\right)\right\|^2 \mid \mathcal{F}_k\right] \leq & \frac{1}{2} n d^2 L^2 t_k^2+6 d L^2\left(1+\sigma_0^2\right)\left\|x^k-1_n \otimes \bar{x}^k\right\|^2\\
&\quad+6 n d\left(1+\sigma_0^2\right)\left(1+\sigma_2^2\right)\left\|\nabla f\left(\bar{x}^k\right)\right\|^2 +6 n d\left(1+\sigma_0^2\right) \sigma_3^2+2 n d \sigma_1^2.
    \end{align*}

\vspace{-25pt}
\end{proof}

\begin{lemma}
\label{lemma4}
Consider the sequence $\left\{m^k\right\}$ generated by Algorithm 1. Under Assumption \ref{ass2}, \ref{ass3}, and \ref{ass4}, for $k \geq 1$, we have
    \begin{align*}
    \mathbb{E}\left\|\bar{m}^k-\nabla \bar{f}^k\right\|^2
    \leq &\; \frac{1+(1-\beta)^2}{2} \mathbb{E}\left\|\bar{m}^{k-1}-\nabla \bar{f}^{k-1}\right\|^2+\left(\frac{32 \beta^2 d}{n}\left(1+\sigma_0^2\right)\left(1+\sigma_2^2\right)+24 \beta^2\right) \mathbb{E}\left\|\nabla f\left(\bar{x}^k\right)\right\|^2 \\
    & +\frac{36\alpha^2(1-\beta)^2 d L^2}{n} \mathbb{E}\left\|\bar{m}^{k-1}\right\|^2 +\frac{36(1-\beta)^2 d L^2}{n^2} \mathbb{E}\left\|x^{k-1}-1_n \otimes \bar{x}^{k-1}\right\|^2 \\
    & +\left(\frac{16 \beta^2 d L^2\left(1+\sigma_0^2\right)}{n^2}+\frac{24 \beta^2 L^2}{n}+\frac{36(1-\beta)^2 d L^2}{n^2}\right) \mathbb{E}\left\|x^k-1_n \otimes \bar{x}^k\right\|^2 \\
    & +\left(2 \beta^2 L^2 d^2+ 16(1-\beta)^2 L^2 +\frac{6(1-\beta)^2 d^2 L^2}{n}+\frac{2(1-\beta)^2(2-\beta) L^2}{\beta}\right) t_k^2 \\
    & +\frac{8 \beta^2 d}{n} \sigma_1^2+\frac{32 \beta^2 d}{n}\left(1+\sigma_0^2\right) \sigma_3^2.
    \end{align*}
\end{lemma}

\begin{proof}
Following the momentum update in Algorithm \ref{algo}, we have
\begin{align*}
\bar{m}^k - \nabla \bar{f}^k&=(1-\beta) (\bar{m}^{k-1} - \nabla \bar{f}^{k-1}) + \beta \underbrace{\left(  \frac{1}{n}\sum_{i=1}^n \hat{\nabla}f_i(x_i^k; u_i^k, \xi_i^k) - \nabla \bar{f}^k\right)}_{:=\bar{v}_k} \\
&\quad + (1-\beta) \underbrace{\left( \frac{1}{n}\sum_{i=1}^n \left( \hat{\nabla} f_i\left(x_i^k ; u_i^k, \xi_i^k\right)-\hat{\nabla} f_i\left(x_i^{k-1} ; u_i^k, \xi_i^k\right) \right) - \left( \nabla \bar{f}^k - \nabla \bar{f}^{k-1}\right)  \right)}_{:=\bar{r}_k}.
\end{align*}
Thus it can be seen that
    \begin{align*}
\mathbb{E}\left[\left\|\bar{m}^k-\nabla \bar{f}^k\right\|^2 \mid \mathcal{F}_k\right]= &\mathbb{E}\left[\left\|(1-\beta)\left(\bar{m}^{k-1}-\nabla \bar{f}^{k-1}\right)+\beta \bar{v}_k+(1-\beta) \bar{r}_k\right\|^2 \mid \mathcal{F}_k\right] \\
    = &\; (1-\beta)^2 \mathbb{E}\left[\left\|\bar{m}^{k-1}-\nabla \bar{f}^{k-1}\right\|^2 \mid \mathcal{F}_k\right]+\mathbb{E}\left[\left\|\beta \bar{v}_k+(1-\beta) \bar{r}_k\right\|^2 \mid \mathcal{F}_k\right] \\
    & +2 \mathbb{E}\left[\left\langle\beta \bar{v}_k+(1-\beta) \bar{r}_k,(1-\beta)\left(\bar{m}^{k-1}-\nabla \bar{f}^{k-1}\right)\right\rangle \mid \mathcal{F}_k\right]. \numberthis\label{4.1}
    \end{align*}
Notice that
    \begin{align*}
    & 2 \mathbb{E}\left[\left\langle\beta \bar{v}_k+(1-\beta) \bar{r}_k,(1-\beta)\left(\bar{m}^{k-1}-\nabla \bar{f}^{k-1}\right)\right\rangle \mid \mathcal{F}_k\right] \\
    = &\; 2 \mathbb{E}_{u^k, \xi^k}\left\langle\beta \bar{v}_k+(1-\beta) \bar{r}_k,(1-\beta)\left(\bar{m}^{k-1}-\nabla \bar{f}^{k-1}\right)\right\rangle \\
    \stackrel{(a)}{=}&\; 2\left\langle \frac{1}{n} \sum_{i=1}^n\left(\nabla f_{i, t_k}\left(x_i^k\right)-\nabla f_i\left(x_i^k\right)\right)-(1-\beta) \frac{1}{n} \sum_{i=1}^n\left(\nabla f_{i, t_k}\left(x_i^{k-1}\right)-\nabla f_i\left(x_i^{k-1}\right)\right),(1-\beta)\left(\bar{m}^{k-1}-\nabla \bar{f}^{k-1}\right)\right\rangle \\
    \leq &\; 2\left\|\frac{1}{n} \sum_{i=1}^n\left(\nabla f_{i, t_k}\left(x_i^k\right)-\nabla f_i\left(x_i^k\right)\right)-(1-\beta) \frac{1}{n} \sum_{i=1}^n\left(\nabla f_{i, t_k}\left(x_i^{k-1}\right)-\nabla f_i\left(x_i^{k-1}\right)\right)\right\| \cdot\left\|(1-\beta)\left(\bar{m}^{k-1}-\nabla \bar{f}^{k-1}\right)\right\| \\
    \stackrel{(b)}{\leq} &\; \frac{2(1-\beta)^2}{\beta(2-\beta)} \left\| \frac{1}{n} \sum_{i=1}^n \left( f_{i, t_k}\left(x_i^k\right)-\nabla f_i\left(x_i^k\right) \right)
    -(1-\beta) \frac{1}{n} \sum_{i=1}^n\left(\nabla f_{i, t_k}\left(x_i^{k-1}\right)-\nabla f_i\left(x_i^{k-1}\right)\right) \right\|^2 \\
    & + \frac{\beta(2-\beta)}{2}\left\|\bar{m}^{k-1}-\nabla \bar{f}^{k-1}\right\|^2\\
    \stackrel{(c)}{\leq}&\; \frac{\beta(2-\beta)}{2}\left\|\bar{m}^{k-1}-\nabla \bar{f}^{k-1}\right\|^2+\frac{2(1-\beta)^2(2-\beta)}{\beta} L^2 t_k^2, \numberthis\label{4.2}
    \end{align*}
    where step (a) is due to Lemma \ref{lemma1}, i.e., $\mathbb{E}_{u^k_i, \xi_i^k} \left[\hat{\nabla} f_i(x_i^k; u_i^k, \xi_i^k) \right] = \nabla f_{i, t_k}(x_i^k)$, step (b) uses the elementary inequality that $2xy\leq \frac{1}{a}x^2 + a y^2$ with $a = \beta(2-\beta)/2$, and step (c) follows from the fact that 
    \begin{align*}
        \left\|\frac{1}{n} \sum_{i=1}^n\left(\nabla f_{i, t_k}\left(x_i^k\right)-\nabla f_i\left(x_i^k\right)\right)-(1-\beta) \frac{1}{n} \sum_{i=1}^n\left(\nabla f_{i, t_k}\left(x_i^{k-1}\right)-\nabla f_i\left(x_i^{k-1}\right)\right)\right\| \leq (2-\beta)Lt_k.
    \end{align*}
    Moreover, the fact used in step (c) can be proved as follows:
    \begin{align*}
    & \left\|\frac{1}{n} \sum_{i=1}^n\left(\nabla f_{i, t_k}\left(x_i^k\right)-\nabla f_i\left(x_i^k\right)\right)-(1-\beta) \frac{1}{n} \sum_{i=1}^n\left(\nabla f_{i, t_k}\left(x_i^{k-1}\right)-\nabla f_i\left(x_i^{k-1}\right)\right)\right\| \\
    \leq &\; \frac{1}{n} \sum_{i=1}^n\left\|\nabla f_{i, t_k}\left(x_i^k\right)-\nabla f_i\left(x_i^k\right)\right\|+(1-\beta) \frac{1}{n} \sum_{i=1}^n\left\|\nabla f_{i, t_k}\left(x_i^{k-1}\right)-\nabla f_i\left(x_i^{k-1}\right)\right\| \\
    \leq &\; L t_k+(1-\beta) L t_k \\
    = &\; (2-\beta) L t_k,
    \end{align*}
    where third line is due to Lemma \ref{lemma1}.
    
    Substituting equation \eqref{4.2} into equation \eqref{4.1} yields
    \begin{align*}
    &\mathbb{E}\left[\left\|\bar{m}^k-\nabla \bar{f}^k\right\|^2 \mid \mathcal{F}_k\right] \\
    \leq &\frac{1+(1-\beta)^2}{2}\left\|\bar{m}^{k-1}-\nabla \bar{f}^{k-1}\right\|^2 
    + \mathbb{E}\left[\left\|\beta \bar{v}_k+(1-\beta) \bar{r}_k\right\|^2 \mid \mathcal{F}_k\right]+\frac{2(1-\beta)^2(2-\beta)}{\beta} L^2 t_k^2 \\
    \leq &\; \frac{1+(1-\beta)^2}{2}\left\|\bar{m}^{k-1}-\nabla \bar{f}^{k-1}\right\|^2 
    + 2\beta^2 \mathbb{E}\left[\left\|\bar{v}_k\right\|^2 \mid \mathcal{F}_k\right] + 2(1-\beta)^2 \mathbb{E}\left[\left\|\bar{r}_k\right\|^2 \mid \mathcal{F}_k\right]\\
    & +\frac{2(1-\beta)^2(2-\beta)}{\beta} L^2 t_k^2. \numberthis\label{4.3}
    \end{align*}
    Next, we bound $\mathbb{E}\left[\left\|\bar{v}_k\right\|^2 \mid \mathcal{F}_k\right]$ and $\mathbb{E}\left[\left\|\bar{r}_k\right\|^2 \mid \mathcal{F}_k\right]$, respectively. For the first term, a simple calculation yields that
    \begin{align*}
\mathbb{E}\left[\left\|\bar{v}_k\right\|^2 \mid \mathcal{F}_k\right]
    = &\; \mathbb{E}\left[\left\|\frac{1}{n} \sum_{i=1}^n \hat{\nabla} f_i\left(x_i^k, u_i^k, z_i^k\right)-\nabla \bar{f}^k\right\|^2 \mid \mathcal{F}_k\right] \\
    = &\; \mathbb{E}\left[\left\|\hat{\nabla} \bar{f}\left(x^k ; u^k, \xi^k\right) - \nabla \bar{f}^k\right\|^2 \mid \mathcal{F}_k\right]\\
    \leq &\; 2 \mathbb{E}\left[\left\|\hat{\nabla} \bar{f}\left(x^k ; u^k, \xi^k\right) \right\|^2 \mid \mathcal{F}_k\right]+2 \mathbb{E}\left[\left\|\nabla \bar{f}^k\right\|^2 \mid \mathcal{F}_k\right] \\
    \leq &\; L^2 d^2 t_k^2+\left(\frac{8 d L^2\left(1+\sigma_0^2\right)}{n^2}+\frac{12 L^2}{n}\right)\left\|x^k-1_n \otimes \bar{x}^k\right\|^2 \\
    & +\left(\frac{16 d}{n}\left(1+\sigma_0^2\right)\left(1+\sigma_2^2\right)+12\right)\left\|\nabla f\left(\bar{x}^k\right)\right\|^2+\frac{16 d}{n}\left(1+\sigma_0^2\right) \sigma_3^2+\frac{4 d}{n} \sigma_1^2, \numberthis\label{4.4}
    \end{align*}
    where the last line follows from Lemma \ref{lemma2} and \eqref{tmp5}. Finally, we turn to control $\mathbb{E}\left[\left\|\bar{r}_k\right\|^2 \mid \mathcal{F}_k\right]$. A direct computation yields that
    \begin{align*}
    & \mathbb{E}_{u^k, \xi^k}\left\|\hat{\nabla} f_i\left(x_i^k ; u_i^k, \xi_i^k\right)-\hat{\nabla} f_i\left(x_i^{k-1} ; u_i^k, \xi_i^k\right)\right\|^2 \\
    \leq &\; 3 \mathbb{E}_{\xi^k} E_{u^k}\left\|d \frac{F_i\left(x_i^k+t_k u_i^k ; \xi_i^k\right)-F_i\left(x_i^k-t_k u_i^k ; \xi_i^k\right)}{2 t_k} u_i^k-d \left\langle \nabla F_i\left(x_i^k ; \xi_i^k\right), u_i^k \right\rangle u_i^k\right\|^2 \\
    & +3 \mathbb{E}_{\xi^k} E_{u^k}\left\|d \frac{F_i\left(x_i^{k-1}+t_k u_i^k ; \xi_i^k\right)-F_i\left(x_i^{k-1}-t_k u_i^k ; \xi_i^k\right)}{2 t_k} u_i^k-d \left\langle \nabla F_i\left(x_i^{k-1} ; \xi_i^k\right), u_i^k \right\rangle u_i^k\right\|^2 \\
    & +3 \mathbb{E}_{\xi^k} E_{u^k}\left\|d \left\langle \nabla F_i\left(x_i^k ; \xi_i^k\right)-\nabla F_i\left(x_i^{k-1} ; \xi_i^k\right), u_i^k \right\rangle u_i^k\right\|^2 \\
    \stackrel{(a)}{\leq}&\; \frac{3}{2} d^2 L^2 t_k^2+3 d E_{\xi^k}\left\|\nabla F_i\left(x_i^k ; \xi_i^k\right)-\nabla F_i\left(x_i^{k-1} ; \xi_i^k\right)\right\|^2 \\
    \stackrel{(b)}{\leq}&\; \frac{3}{2} d^2 L^2 t_k^2+3 d L^2\left\|x_i^k-x_i^{k-1}\right\|^2, \numberthis\label{4.5}
    \end{align*}
    where step (a) is due to Lemma \ref{lemma1} and step (b) follows from the $L$-smoothness of $F_i\left(\cdot;\xi_i^k\right)$. Thus it can be seen that
    \begin{align*}
    & \mathbb{E}\left[\left\|\bar{r}_k\right\|^2 \mid \mathcal{F}_k\right] \\
    = &\; \mathbb{E}\left[\left\|\frac{1}{n} \sum_{i=1}^n\left(\hat{\nabla} f_i\left(x_i^k ; u_i^k, \xi_i^k\right)-\hat{\nabla} f_i\left(x_i^{k-1} ; u_i^k, \xi_i^k\right)+\nabla f_i\left(x_i^{k-1}\right)-\nabla f_i\left(x_i^k\right)\right)\right\|^2 \mid \mathcal{F}_k\right] \\
    = &\; \mathbb{E}_{u^k, \xi^k} \Bigg\| \frac{1}{n} \sum_{i=1}^n {\left[\hat{\nabla} f_i\left(x_i^k ; u_i^k, \xi_i^k\right)-\hat{\nabla} f_i\left(x_i^{k-1} ; u_i^k, \xi_i^k\right)-\left(\nabla f_{i, t_k}\left(x_i^k\right)-\nabla f_{i, t_k}\left(x_i^{k-1}\right)\right)\right.}\\
    &\left.+\left(\nabla f_{i,t_k}\left(x_i^k\right)-\nabla f_{i, t_k}\left(x_i^{k-1}\right)\right)-\left(\nabla f_i\left(x_i^k\right)-\nabla f_i\left(x_i^{k-1}\right)\right)\right] \Bigg\|^2 \\
    \leq &\; 2 \mathbb{E}_{u^k, \xi^k}\left\|\frac{1}{n} \sum_{i=1}^n\left[\hat{\nabla} f_i\left(x_i^k ; u_i^k, \xi_i^k\right)-\hat{\nabla} f_i\left(x_i^{k-1} ; u_i^k, \xi_i^k\right)-\left(\nabla f_{i, t_k}\left(x_i^k\right)-\nabla f_{i, t_k}\left(x_i^{k-1}\right)\right)\right]\right\|^2 \\
    &+2\left\|\frac{1}{n} \sum_{i=1}^n\left[\left(\nabla f_{i, t_k}\left(x_i^k\right)-\nabla f_{i, t_k}\left(x_i^{k-1}\right)\right)-\left(\nabla f_i\left(x_i^k\right)-\nabla f_i\left(x_i^{k-1}\right)\right)\right]\right\|^2 \\
    \leq &\; \frac{2}{n^2} \mathbb{E}_{u^k, \xi^k} \sum_{i=1}^n\left\|\hat{\nabla} f_i\left(x_i^k ; u_i^k, \xi_i^k\right)-\hat{\nabla} f_i\left(x_i^{k-1} ; u_i^k, \xi_i^k\right)-\left(\nabla f_{i, t_k}\left(x_i^k\right)-\nabla f_{i, t_k}\left(x_i^{k-1}\right)\right)\right\|^2 \\
    & +4\left\|\frac{1}{n} \sum_{i=1}^n \nabla f_{i, t_k}\left(x_i^k\right)-\nabla f_i\left(x_i^k\right)\right\|^2+4\left\|\frac{1}{n} \sum_{i=1}^n \nabla f_{i, t_k}\left(x_i^{k-1}\right)-\nabla f_i\left(x_i^{k-1}\right)\right\|^2 \\
    \leq &\; \frac{2}{n^2} \sum_{i=1}^n \mathbb{E}_{u^k, \xi^k}\left\|\hat{\nabla} f_i\left(x_i^k ; u_i^k, \xi_i^k\right)-\hat{\nabla} f_i\left(x_i^{k-1} ; u_i^k, \xi_i^k\right)\right\|^2+8 L^2 t_k^2\\
    \stackrel{(a)}{\leq} &\; \frac{2}{n^2} \sum_{i=1}^n\left(\frac{3}{2} d^2 L^2 t_k^2+3 d L^2\left\|x_i^k-x_i^{k-1}\right\|^2\right)+8 L^2 t_k^2 \\
    = &\; \frac{6 d L^2}{n^2}\left\|x^k-x^{k-1}\right\|^2+\left(8 L^2+\frac{3 d^2 L^2}{n}\right) t_k^2 \\
    \leq &\; \frac{18 d L^2}{n^2}\left\|x^k-1_n \otimes \bar{x}^k\right\|^2+\frac{18 d L^2}{n^2}\left\|x^{k-1}-1_n \otimes x^{k-1}\right\|^2+\frac{18 \alpha^2 d L^2}{n}\left\|\bar{m}^{k-1}\right\|^2+\left(8 L^2+\frac{3 d^2 L^2}{n}\right) t_k^2,\numberthis\label{4.6}
    \end{align*}
    where step (a) is due to \eqref{4.5}. Substituting equations \eqref{4.4} and \eqref{4.6} into equation \eqref{4.3} and taking a total expectation, we get
    \begin{align*}
    \mathbb{E}\left\|\bar{m}^k-\nabla \bar{f}^k\right\|^2
    \leq & \frac{1+(1-\beta)^2}{2} \mathbb{E}\left\|\bar{m}^{k-1}-\nabla \bar{f}^{k-1}\right\|^2 \\
    & +\left(\frac{32 \beta^2 d}{n}\left(1+\sigma_0^2\right)\left(1+\sigma_2^2\right)+24 \beta^2\right) \mathbb{E}\left\|\nabla f\left(\bar{x}^k\right)\right\|^2 \\
    & +\frac{36\alpha^2(1-\beta)^2 d L^2}{n} \mathbb{E}\left\|\bar{m}^{k-1}\right\|^2 +\frac{36(1-\beta)^2 d L^2}{n^2} \mathbb{E}\left\|x^{k-1}-1_n \otimes \bar{x}^{k-1}\right\|^2 \\
    & +\left(\frac{16 \beta^2 d L^2\left(1+\sigma_0^2\right)}{n^2}+\frac{24 \beta^2 L^2}{n}+\frac{36(1-\beta)^2 d L^2}{n^2}\right) \mathbb{E}\left\|x^k-1_n \otimes \bar{x}^k\right\|^2 \\
    & +\left(2 \beta^2 L^2 d^2+ 16(1-\beta)^2 L^2 +\frac{6(1-\beta)^2 d^2 L^2}{n}+\frac{2(1-\beta)^2(2-\beta) L^2}{\beta}\right) t_k^2 \\
    & +\frac{8 \beta^2 d}{n} \sigma_1^2+\frac{32 \beta^2 d}{n}\left(1+\sigma_0^2\right) \sigma_3^2,
    \end{align*}
which completes the proof.
\end{proof}

\begin{lemma}
\label{lemma5} Consider the sequence $\left\{m^k\right\}$ generated by Algorithm 1. For $k \geq 1$, we have
    \begin{align*}
    \mathbb{E}\left\|m^k-\nabla f\left(x^k\right)\right\|^2 
    \leq &\; \frac{1+(1-\beta)^2}{2} \mathbb{E}\left\|m^{k-1}-\nabla f\left(x^{k-1}\right)\right\|^2 +36 \beta^2 n d\left(1+\sigma_0^2\right)\left(1+\sigma_2^2\right) \mathbb{E}\left\|\nabla f\left(\bar{x}^k\right)\right\|^2 \\
    & +\left(36 \beta^2 d L^2\left(1+\sigma_0^2\right)+48(1-\beta)^2 d L^2\right) \mathbb{E}\left\|x^k-1_n \otimes \bar{x}^k\right\|^2 \\
    & +48(1-\beta)^2 d L^2 \mathbb{E}\left\|x^{k-1}-1_n \otimes \bar{x}^{k-1}\right\|^2 +48 \alpha^2(1-\beta)^2 d L^2 d n \mathbb{E}\left\|\bar{m}^{k-1}\right\|^2 \\
    & +\left(\left(2 \beta^2+6(1-\beta)^2\right) n d^2 L^2+\frac{8 n L^2(1-\beta)}{\beta(2-\beta)}\right) t_k^2+36 \beta^2 n d\left(1+\sigma_0^2\right) \sigma_3^2+8 \beta^2 n d \sigma_1^2.
    \end{align*}
\end{lemma}

\begin{proof}
Recalling that
\begin{align*}
    m^k & =\beta \hat{\nabla} f\left(x^k ; u^k, \xi^k\right)+(1-\beta)\left(m^{k-1}+\hat{\nabla} f\left(x^k ; u^k, \xi^k\right)-\hat{\nabla} f\left(x^{k-1} ; u^k, \xi^k\right)\right),
\end{align*}
we have
    \begin{align*}
    & \mathbb{E}\left[\left\|m^k-\nabla f\left(x^k\right)\right\|^2 \mid \mathcal{F}_k\right] \\
    = & \mathbb{E}\left[\bigg\| \beta \hat{\nabla} f\left(x^k ; u^k, \xi^k\right)+(1-\beta)\left(m^{k-1}+\hat{\nabla} f\left(x^k ; u^k, \xi^k\right)-\hat{\nabla} f\left(x^{k-1} ; u^k, \xi^k\right)\right)-\nabla f\left(x^k\right) \bigg\|^2 \mid \mathcal{F}_k\right] \\
    = & \mathbb{E}\left[\bigg\| \beta\left(\hat{\nabla} f\left(x^k ; u^k, \xi^k\right)-\nabla f\left(x^k\right) \right)+(1-\beta)\left(m^{k-1}-\nabla f\left(x^{k-1}\right)\right)\right. \\
    & \left.+(1-\beta)\left(\hat{\nabla} f\left(x^k ; u^k, \xi^k\right)-\hat{\nabla} f\left(x^{k-1} ; u^k, \xi^k\right)+\nabla f\left(x^{k-1}\right)-\nabla f \left( x^k\right)\right) \bigg\|^2 \mid \mathcal{F}_k\right] \\
    = & (1-\beta)^2 \mathbb{E}\left[\left\|m^{k-1}-\nabla f\left(x^{k-1}\right)\right\|^2 \mid \mathcal{F}_k\right]+\mathbb{E}\left[\left\|\beta v_k+(1-\beta) r_k\right\|^2 \mid \mathcal{F}_k\right] \\
    & +2 \mathbb{E}\left[\left\langle\beta v_k+(1-\beta) r_k,(1-\beta)\left(m^{k-1}-\nabla f\left(x^{k-1}\right)\right)\right\rangle \mid \mathcal{F}_k\right], \numberthis\label{5.1}
    \end{align*}
    where 
    \begin{align*}
    v_k &= \hat{\nabla} f\left(x^k ; u^k, \xi^k\right)-\nabla f\left(x^k\right),\\
    r_k &= \hat{\nabla} f\left(x^k ; u^k, \xi^k\right)-\hat{\nabla} f\left(x^{k-1} ; u^k, \xi^k\right)+\nabla f\left(x^{k-1}\right)-\nabla f\left(x^k\right).
    \end{align*}

    \noindent The last term of \eqref{5.1} can be bounded as follows:
    \begin{align*}
    & 2 \mathbb{E}\left[\left\langle\beta v_k+(1-\beta) r_k,(1-\beta)\left(m^{k-1}-\nabla f\left(x^{k-1}\right)\right)\right\rangle \mid \mathcal{F}_k\right] \\
    = & 2\left\langle\nabla f_{t_k}\left(x^k\right)-\nabla f\left(x^k\right)+(1-\beta)\left(\nabla f_{t_k}\left(x^{k-1}\right)-\nabla f\left(x^{k-1}\right)\right),(1-\beta)\left(m^{k-1}-\nabla f\left(x^{k-1}\right)\right)\right\rangle \\
    \leq & 2(1-\beta)\left\|\nabla f_{t_k}\left(x^k\right)-\nabla f\left(x^k\right)+(1-\beta)\left(\nabla f_{t_k}\left(x^{k-1}\right)-\nabla f\left(x^{k-1}\right)\right)\right\| \cdot\left\|m^{k-1}-\nabla f\left(x^{k-1}\right)\right\| \\
    \leq & \frac{\beta(2-\beta)}{2}\left\|m^{k-1}-\nabla f\left(x^{k-1}\right)\right\|^2+\frac{2(1-\beta)}{\beta(2-\beta)}\left\|\nabla f_{t_k}\left(x^k\right)-\nabla f\left(x^k\right)+(1-\beta)\left(\nabla f_{t_k}\left(x^{k-1}\right)-\nabla f\left(x^{k-1}\right)\right)\right\|^2 \\
    \leq & \frac{\beta(2-\beta)}{2}\left\|m^{k-1}-\nabla f\left(x^{k-1}\right)\right\|^2+\frac{8 n L^2(1-\beta)}{\beta(2-\beta)} t_k^2, \numberthis \label{5.2}
    \end{align*}
    where the last inequality is due to
    \begin{align*}
    & \left\|\nabla f_{t_k}\left(x^k\right)-\nabla f\left(x^k\right)+(1-\beta)\left(\nabla f_{t_k}\left(x^{k-1}\right)-\nabla f\left(x^{k-1}\right)\right)\right\|^2 \\
    \leq & 2\left\|\nabla f_{t_k}\left(x^k\right)-\nabla f\left(x^k\right)\right\|^2+2(1-\beta)^2\left\|\nabla f_{t_k}\left(x^{k-1}\right)-\nabla f\left(x^{k-1}\right)\right\|^2 \\
    \leq & 2 n L^2 t_k^2+2(1-\beta)^2 n L^2 t_k^2 \\
    \leq & 4 n L^2 t_k^2.
    \end{align*}
    Thus, substituting \eqref{5.2} into  \eqref{5.1} gives that
    \begin{align*}
    &\mathbb{E}\left[\left\|m^k-\nabla f\left(x^k\right)\right\|^2 \mid \mathcal{F}_k\right] \\
    \leq& \frac{1+(1-\beta)^2}{2}\left\|m^{k-1}-\nabla f\left(x^{k-1}\right)\right\|^2+2 \beta^2 \mathbb{E}\left[\left\|\hat{\nabla}f\left(x^k ; u^k, \xi^k\right)-\nabla f\left(x^k\right)\right\|^2 \mid \mathcal{F}_k\right] \\
    & +2(1-\beta)^2 \mathbb{E}\left[\left\|\hat{\nabla} f\left(x^k ; u^k, \xi^k\right)-\hat{\nabla} f\left(x^{k-1} ; u^k, \xi^k\right)+\nabla f\left(x^{k-1}\right)-\nabla f\left(x^k\right)\right\|^2 \mid \mathcal{F}_k\right] \\
    & +\frac{8 n L^2(1-\beta)}{\beta(2-\beta)} t_k^2. \numberthis\label{5.3}
    \end{align*}
    
    \noindent According to Lemma \ref{lemma3}, it can be shown that
    \begin{align*}
    & \mathbb{E}\left[\left\|\hat{\nabla} f\left(x^k ; u^k, \xi^k\right)-\nabla f\left(x^k\right)\right\|^2 \mid \mathcal{F}_k\right] \\
    \leq & \; 2 \mathbb{E}\left[\left\|\hat{\nabla} f\left(x^k ; u^k, \xi^k\right)\right\|^2 \mid \mathcal{F}_k\right]+2\left\|\nabla f\left(x^k\right)\right\|^2 \\
    \leq & \;  n d^2 L^2 t_k^2+12 d L^2\left(1+\sigma_0^2\right)\left\|x^k-1_n \otimes \bar{x}^k\right\|^2+12 n d\left(1+\sigma_0^2\right)\left(1+\sigma_2^2\right)\left\|\nabla f\left(\bar{x}^k\right)\right\|^2 \\
    & +12 n d\left(1+\sigma_0^2\right) \sigma_3^2+4 n d \sigma_1^2+2\left(3 L^2\left\|x^k-1_n \otimes \bar{x}^k\right\|^2+3 n\left(1+\sigma_2^2\right)\left\|\nabla f\left(\bar{x}^k\right)\right\|^2+3 n \sigma_3^2\right) \\
    \leq & \;  n d^2 L^2 t_k^2+18 d L^2\left(1+\sigma_0^2\right)\left\|x^k-1_n \otimes \bar{x}^k\right\|^2+18 n d\left(1+\sigma_0^2\right)\left(1+\sigma_2^2\right)\left\|\nabla f\left(\bar{x}^k\right)\right\|^2 \\
    & +18 n d\left(1+\sigma_0^2\right) \sigma_3^2+4 n d \sigma_1^2. \numberthis\label{5.4}
    \end{align*}

    \noindent Moreover, we have
    \begin{align*}
    & \mathbb{E}\left[\left\|\hat{\nabla} f\left(x^k ; u^k, \xi^k\right)-\hat{\nabla} f\left(x^{k-1} ; u^k, \xi^k\right)+\nabla f\left(x^{k-1}\right)-\nabla f\left(x^k\right)\right\|^2 \mid \mathcal{F}_k\right] \\
    = & \; \sum_{i=1}^n \mathbb{E}\left[\left\|\hat{\nabla} f_i\left(x_i^k ; u_i^k, \xi_i^k\right)-\hat{\nabla} f_i\left(x_i^{k-1} ; u_i^k, \xi_i^k\right)+\nabla f_i\left(x_i^{k-1}\right)-\nabla f_i\left(x_i^k\right)\right\|^2 \mid \mathcal{F}_k\right] \\
    \leq & \; 2 \sum_{i=1}^n\left(\frac{3}{2} d^2 L^2 t_k^2+3 d L^2 \left\|x_i^k-x_i^{k-1}\right\|^2\right)+2 \sum_{i=1}^n L^2\left\|x_i^k-x_i^{k-1}\right\|^2 \\
    \leq & \; 3 n d^2 L^2 t_k^2+ 8 d L^2\left\|x^k-x^{k-1}\right\|^2 \\
    \leq & \; 3 n d^2 L^2 t_k^2+24 d L^2\left\|x^k-1_n \otimes \bar{x}^k\right\|^2+24 d L^2 \left\|x^{k-1}-1_n \otimes \bar{x}^{k-1}\right\|^2+24 \alpha^2 d L^2 n\left\|\bar{m}^{k-1}\right\|^2, \numberthis\label{5.5}
    \end{align*}
    where the first inequality follows from \eqref{4.5}.

    Plugging \eqref{5.4} and \eqref{5.5} into \eqref{5.3}, and taking the total expectation on both sides, we eventually obtain that 
    \begin{align*}
    & \mathbb{E}\left\|m^k-\nabla f\left(x^k\right)\right\|^2 \\
    \leq &\; \frac{1+(1-\beta)^2}{2} \mathbb{E}\left\|m^{k-1}-\nabla f\left(x^{k-1}\right)\right\|^2+36 \beta^2 n d\left(1+\sigma_0^2\right)\left(1+\sigma_2^2\right) \mathbb{E}\left\|\nabla f\left(\bar{x}^k\right)\right\|^2 \\
    & +\left(36 \beta^2 d L^2\left(1+\sigma_0^2\right)+48(1-\beta)^2 d L^2 \right) \mathbb{E}\left\|x^k-1_n \otimes \bar{x}^k\right\|^2+48(1-\beta)^2 d L^2 \mathbb{E}\left\|x^{k-1}-1_n \otimes \bar{x}^{k-1}\right\|^2 \\
    & +48 \alpha^2(1-\beta)^2 d L^2 n \mathbb{E}\left\|\bar{m}^{k-1}\right\|^2 \\
    & +\left(\left(2 \beta^2+6(1-\beta)^2\right) n d^2 L^2+\frac{8 n L^2(1-\beta)}{\beta(2-\beta)}\right) t_k^2+36 \beta^2 n d\left(1+\sigma_0^2\right) \sigma_3^2+8 \beta^2 n d \sigma_1^2.
    \end{align*}

\vspace{-30pt}
\end{proof}

\begin{lemma}
\label{lemma6}
Consider the sequence $\left\{x^k\right\}$ generated by Algorithm 1. Under Assumption \ref{ass1}, for $k \geq 1$, we have
    $$
    \mathbb{E}\left\|x^{k+1}-1_n \otimes \bar{x}^{k+1}\right\|^2
    \leq \frac{1+\rho^2}{2} \mathbb{E}\left\|x^k-1_n \otimes \bar{x}^k\right\|^2+\frac{\alpha^2\left(1+\rho^2\right) \rho^2}{1-\rho^2} \mathbb{E}\left\|g^{k+1}-1_n \otimes \bar{g}^{k+1}\right\|^2.
    $$
\end{lemma}

\begin{proof}
Recall that $x^{k+1}=(W\otimes I_d)(x^k - \alpha g^{k+1})$. A simple calculation yields that
    \begin{align*}
    \mathbb{E}\left\|x^{k+1}-1_n \otimes \bar{x}^{k+1}\right\|^2 & =\mathbb{E}\left\|\left(I_{n d}-\frac{1_n 1_n^T}{n} \otimes I_d\right) x^{k+1}\right\|^2 \\
    & =\mathbb{E}\left\|\left(I_{n d}-\frac{1_n 1_n^T}{n} \otimes I_d\right)\left(W \otimes I_d\right)\left(x^k-\alpha g^{k+1}\right)\right\|^2 \\
    & =\mathbb{E}\left\|\left(\left(W-\frac{1_n 1_n^T}{n}\right) \otimes I_d\right)\left(x^k-\alpha g^{k+1}\right)\right\|^2 \\
    & = \mathbb{E}\left\|\left(\left(W-\frac{1_n 1_n^T}{n}\right) \otimes I_d\right)\left(x^k- 1_n\otimes \bar{x}^k - \alpha\left( g^{k+1} - 1_n\otimes \bar{g}^{k+1}\right)\right)\right\|^2 \\
    &\leq \rho^2 \cdot \mathbb{E}\left\|x^k- 1_n\otimes \bar{x}^k - \alpha\left( g^{k+1} - 1_n\otimes \bar{g}^{k+1}\right)\right\|^2 \\
    &\leq \rho^2\left( 1+ \frac{1-\rho^2}{2\rho^2}\right)\mathbb{E}\left\|x^k- 1_n\otimes \bar{x}^k \right\|^2 + \rho^2 \left( 1+\frac{2\rho^2}{1-\rho^2} \right)\alpha^2 \mathbb{E}\left\| g^{k+1} - 1_n\otimes \bar{g}^{k+1}\right\|^2\\
    & =\frac{1+\rho^2}{2} \mathbb{E}\left\|x^k-1_n \otimes \bar{x}^k\right\|^2+\frac{\alpha^2\left(1+\rho^2\right) \rho^2}{1-\rho^2} \mathbb{E}\left\|g^{k+1}-1_n \otimes \bar{g}^{k+1}\right\|^2,
    \end{align*}
where the third line is due to Assumption \ref{ass1}, the fourth line uses the fact that $((W - n^{-1}1_n1_n^\top)\otimes I_d)(1_n\otimes a)=0$ for any vector $a\in \mathbb{R}^d$, and the last line follows from the element inequality that $\|a+b\|^2\leq (1+\gamma) \|a\|^2 + (1+\gamma^{-1})\|b\|^2$ with $\gamma=(1-\rho^2)/(2\rho^2)$.
\end{proof}

\begin{lemma}
\label{lemma7}
Consider the sequence $\{g^k\}$ generated by Algorithm \ref{algo}. Suppose $\alpha \leq \frac{1-\rho^2}{12 \sqrt{2} d L \rho^2}$ and $\beta \leq \frac{\sqrt{d}}{\sqrt{27\left(1-\rho^2\right)\left(1+\sigma_0^2\right)}}$. Under Assumption \ref{ass1}, \ref{ass2} and \ref{ass3}, for $k \ge 1$, we have
    \begin{align*}
    \mathbb{E}\left\|g^{k+1}-1_n \otimes \bar{g}^{k+1}\right\|^2
    \leq &\; \frac{3+\rho^2}{4} \mathbb{E}\left\|g^k-1_n \otimes \bar{g}^k\right\|^2+\frac{36 \alpha^2 n d^2 L^2 \rho^2}{1-\rho^2} \mathbb{E}\left\|\bar{m}^{k-1}\right\|^2 \\
    & +\frac{9 \beta^2 \rho^2}{1-\rho^2} \mathbb{E}\left\|m^{k-1}-\nabla f\left(x^{k-1}\right)\right\|^2+\frac{110 d^2 L^2 \rho^2}{1-\rho^2} \mathbb{E}\left\|x^{k-1}-1_n \otimes \bar{x}^{k-1}\right\|^2 \\
    & +54 \beta^2 n d \rho^2\left(1+\sigma_0^2\right)\left(1+\sigma_2^2\right) \mathbb{E}\left\|\nabla f\left(\bar{x}^{k-1}\right)\right\|^2+\left(9 n d^2 L^2 \rho^2+\frac{6 \rho^2 \beta^2 L^2}{1-\rho^2}\right) t_k^2 \\
    & +54 \beta^2 n d \rho^2\left(1+\sigma_0^2\right) \sigma_3^2+12 \beta^2 n d \rho^2 \sigma_1^2. \\
    \end{align*}
\end{lemma}

\begin{proof}
Recall that
\begin{align*}
    g^{k+1} = (W\otimes I_d)(g^k + m^k - m^{k-1}).
\end{align*}
A simple computation yields that
\begin{align*}
    1_n\otimes \bar{g}^{k+1} = 1_n \otimes \left(\frac{1}{n}1_n^\top \otimes I_d\right) g^{k+1} = \frac{1}{n} \left(1_n1_n^\top \otimes I_d \right) g^{k+1}.
\end{align*}
Thus we have
\begin{align*}
    & \mathbb{E}\left[\left\|g^{k+1}-1_n \otimes \bar{g}^{k+1}\right\|^2 \mid \mathcal{F}_k\right]\\
    =&\; \mathbb{E}\left[\left\|(W\otimes I_d)(g^k +m^k - m^{k-1})-\frac{1}{n}(1_n1_n^\top \otimes I_d) (W\otimes I_d)(g^k +m^k - m^{k-1})\right\|^2 \mid \mathcal{F}_k\right]\\
    \stackrel{(a)}{=}&\; \mathbb{E}\left[\left\|(W\otimes I_d)(g^k +m^k - m^{k-1})-\frac{1}{n}(1_n1_n^\top \otimes I_d) (g^k +m^k - m^{k-1})\right\|^2 \mid \mathcal{F}_k\right]\\
    = &\; \mathbb{E}\left[\left\|\left(\left(W-\frac{1_n 1_n^\top}{n}\right) \otimes I_d\right)\left(g^k+m^k-m^{k-1}\right)\right\|^2 \mid \mathcal{F}_k\right] \\
    \stackrel{(b)}{=} &\; \mathbb{E}\left[\left\| \left(\left(W-\frac{1_n 1_n^\top}{n}\right) \otimes I_d \right)\left(g^k-1_n \otimes \bar{g}^k\right)\right\|^2 \mid \mathcal{F}_k\right]\\
    & \quad +\mathbb{E}\left[\left\|\left(\left(W-\frac{1_n 1_n^\top}{n}\right) \otimes I_d\right)\left(m^k-m^{k-1}\right)\right\|^2 \mid \mathcal{F}_k\right] \\
    & \quad +2 \mathbb{E}\left[\left\langle \left(\left(W-\frac{1_n 1_n^\top}{n}\right) \otimes I_d\right)\left(g^k-1_n \otimes \bar{g}^k\right),\left(\left(W-\frac{1_n 1_n^\top}{n}\right) \otimes I_d\right)\left(m^k-m^{k-1}\right)\right\rangle \mid \mathcal{F}_k\right] \\
    \leq &\; \rho^2 \mathbb{E}\left[\left\|g^k-1_n \otimes \bar{g}^k\right\|^2 \mid \mathcal{F}_k\right]\\
    & \quad +\underbrace{2 \mathbb{E}\left[\left\langle\left(\left(W-\frac{1_n 1_n^\top}{n}\right) \otimes I_d\right)\left(g^k-1_n \otimes \bar{g}^k\right),\left(\left(W-\frac{1_n 1_n^\top}{n}\right) \otimes I_d\right)\left(m^k-m^{k-1}\right)\right\rangle \mid \mathcal{F}_k\right]}_{:=T_1}\\
    &\quad+\rho^2 \underbrace{\mathbb{E}\left[\left\|m^k-m^{k-1}\right\|^2 \mid \mathcal{F}_k\right]}_{:=T_2},\numberthis\label{7.1}
    \end{align*}
    where step (a) is due to Assumption \ref{ass1}, i.e., $1_n^\top W = 1_n^\top$, step (b) follows from the fact that $((W -n^{-1}1_n1_n^\top)\otimes I_d)(1_n \otimes a) = 0$ for any $a\in\mathbb{R}^d$. Now, we turn to control $T_1$ and $T_2$ in \eqref{7.1}, respectively.
    \begin{itemize}
        \item Bounding $T_1$. For sake of clarity, we denote $\widetilde{W}=(W-n^{-1}1_n1_n^\top)\otimes I_d$. Then we have
    \begin{align*}
        T &= 2 \mathbb{E}\left[\left\langle\widetilde{W}\left(g^k-1_n \otimes \bar{g}^k\right),\widetilde{W}\left(m^k-m^{k-1}\right)\right\rangle \mid \mathcal{F}_k\right] \\
        &=2 \left\langle\widetilde{W}\left(g^k-1_n \otimes \bar{g}^k\right),\widetilde{W}\left(\mathbb{E}\left[m^k\mid \mathcal{F}_k\right]-m^{k-1}\right)\right\rangle\\
        &\leq 2\left\|\widetilde{W}\left(g^k-1_n \otimes \bar{g}^k\right)\right\|\cdot \left\|\widetilde{W}\left(\mathbb{E}\left[m^k\mid \mathcal{F}_k\right]-m^{k-1}\right)\right\|\\
        &\leq 2\cdot \rho \left\|g^k-1_n \otimes \bar{g}^k\right\|\cdot \rho\left\|\mathbb{E}\left[m^k\mid \mathcal{F}_k\right]-m^{k-1}\right\|\\
        &\stackrel{(a)}{\leq} \frac{1-\rho^2}{2}\cdot \rho^2\left\|g^k-1_n \otimes \bar{g}^k\right\|^2 + \frac{2}{1-\rho^2} \cdot  \rho^2\left\|\mathbb{E}\left[m^k\mid \mathcal{F}_k\right]-m^{k-1}\right\|^2\\
        &\stackrel{(b)}{\leq}\frac{1-\rho^2}{2}\left\|g^k-1_n \otimes \bar{g}^k\right\|^2 + \frac{2}{1-\rho^2} \cdot  \rho^2\left\|\mathbb{E}\left[m^k\mid \mathcal{F}_k\right]-m^{k-1}\right\|^2\\
        &\stackrel{(c)}{\leq}\frac{1-\rho^2}{2}\left\|g^k-1_n \otimes \bar{g}^k\right\|^2 + \frac{2\rho^2}{1-\rho^2}\left(3\beta^2 \left\| m^{k-1} - \nabla f(x^{k-1})\right\|^2 + 3L^2\left\|x^k - x^{k-1}\right\|^2 + 3\beta^2L^2 t_k^2\right)\\
        &= \frac{1-\rho^2}{2}\left\|g^k-1_n \otimes \bar{g}^k\right\|^2+\frac{6 \rho^2 L^2}{1-\rho^2}\left\|x^k-x^{k-1}\right\|^2+\frac{6 \rho^2 \beta^2}{1-\rho^2}\left\|m^{k-1}-\nabla f\left(x^{k-1}\right)\right\|^2+\frac{6 \rho^2 \beta^2 L^2}{1-\rho^2} t_k^2,\numberthis\label{7.2}
    \end{align*}
    where step (a) uses the elementary inequality that $2xy\leq a^{-1}x^2 + ay^2$ with $a=2/(1-\rho^2)$, step (b) follows from $\rho<1$, and step (c) is due to the fact that $\left\|\mathbb{E}\left[m^k\mid \mathcal{F}_k\right] - m^{k-1}\right\|^2\leq 3\beta^2 \left\| m^{k-1} - \nabla f(x^{k-1})\right\|^2 + 3L^2\left\|x^k - x^{k-1}\right\|^2 + 3\beta^2L^2 t_k^2$. Moreover, the fact used in step (c) can be proved as follows: recalling that $m^k = \beta \hat{\nabla}f(x^k;u^k,\xi^k)+ (1-\beta)(m^{k-1}+\hat{\nabla}f(x^k;u^k,\xi^k)-\hat{\nabla}f(x^{k-1};u^k,\xi^k))$, we have
    \begin{align*}
    &\quad \left\|\mathbb{E}\left[m^k\mid \mathcal{F}_k\right] - m^{k-1}\right\|^2 \\
    &= \left\|\beta \nabla f_{t_k}(x^k) + (1-\beta)(m^{k-1} + \nabla f_{t_k}(x^k) - \nabla f_{t_k}(x^{k-1})) - m^{k-1}\right\|^2\\
        &=\left\|-\beta m^{k-1}+ \nabla f_{t_k}(x^k) - (1-\beta) \nabla f_{t_k}(x^{k-1}) \right\|^2\\
        &=\left\|-\beta \left(m^{k-1} - \nabla f(x^{k-1})\right)+\left( \nabla f_{t_k}(x^k) -\nabla f_{t_k}(x^{k-1})\right) +\beta\left(\nabla f_{t_k}(x^{k-1}) -\nabla f(x^{k-1})\right)\right\|^2\\
        &\leq 3\beta^2 \left\| m^{k-1} - \nabla f(x^{k-1})\right\|^2 + 3\left\|\nabla f_{t_k}(x^k) -\nabla f_{t_k}(x^{k-1})\right\|^2 + 3\beta^2\left\|\nabla f_{t_k}(x^{k-1}) -\nabla f(x^{k-1})\right\|^2\\
        &\leq 3\beta^2 \left\| m^{k-1} - \nabla f(x^{k-1})\right\|^2 + 3L^2\left\|x^k - x^{k-1}\right\|^2 + 3\beta^2L^2 t_k^2,
    \end{align*}
    where $\nabla f_{t_k}(x^k)=\begin{bmatrix}
        \nabla f_{1,t_k}(x_1^k)^\top &\cdots & \nabla f_{n,t_k}(x_n^k)^\top
    \end{bmatrix}^\top$ and the last line follows from Lemma \ref{lemma1}.

    \item Bounding $T_2$. 
    A simple computation yields that
    \begin{align*}
    T_2 =&\; \mathbb{E}\bigg[\bigg\|\hat{\nabla} f\left(x^k ; u^k, \xi^k\right)-\hat{\nabla} f\left(x^{k-1} ; u^k, \xi^k\right)-\beta\left(m^{k-1}-\nabla f\left(x^{k-1}\right)\right) \\
    &\qquad+\beta\left(\hat{\nabla} f\left(x^{k-1} ; u^k, \xi^k\right)-\nabla f\left(x^{k-1}\right)\right)\bigg\|^2 \mid \mathcal{F}_k\bigg] \\
    \leq &\; 3 \mathbb{E}\left[\left\|\hat{\nabla} f\left(x^k ; u^k, \xi^k\right)-\hat{\nabla} f\left(x^{k-1} ; u^k, \xi^k\right)\right\|^2 \mid \mathcal{F}_k\right]+3 \beta^2 \mathbb{E}\left[\left\|m^{k-1}-\nabla f\left(x^{k-1}\right)\right\|^2 \mid \mathcal{F}_k\right] \\
    & +3 \beta^2 \mathbb{E}\left[\left\|\hat{\nabla} f\left(x^{k-1} ; u^k, \xi^k\right)-\nabla f\left(x^{k-1}\right)\right\|^2 \mid F_k\right] \\
    \leq &\; \frac{9}{2} n d^2 L^2 t_k^2+9 d L^2 \left\|x^k-x^{k-1}\right\|^2+3 \beta^2\left\|m^{k-1}-\nabla f\left(x^{k-1}\right)\right\|^2 \\
    & +3 \beta^2 n d^2 L^2 t_k^2+54 \beta^2 d L^2\left(1+\sigma_0^2\right)\left\|x^{k-1}-1_n \otimes \bar{x}^{k-1}\right\|^2+54 \beta^2 n d\left(1+\sigma_0^2\right)\left(1+\sigma_2^2\right)\left\|\nabla f\left(\bar{x}^{k-1}\right)\right\|^2 \\
    & +54 \beta^2 n d\left(1+\sigma_0^2\right) \sigma_3^2+12 \beta^2 n d \sigma_1^2 \\
    \leq &\; 8 n d^2 L^2 t_k^2+9 d L^2\left\|x^k-x^{k-1}\right\|^2+3 \beta^2\left\|m^{k-1}-\nabla f\left(x^{k-1}\right)\right\|^2 \\
    & +54 \beta^2 d L^2\left(1+\sigma_0^2\right)\left\|x^{k-1}-1_n \otimes \bar{x}^{k-1}\right\|^2+54 \beta^2 n d\left(1+\sigma_0^2\right)\left(1+\sigma_2^2\right)\left\|\nabla f\left(\bar{x}^{k-1}\right)\right\|^2 \\
    & +54 \beta^2 n d\left(1+\sigma_0^2\right) \sigma_3^2+12 \beta^2 n d \sigma_1^2,\numberthis\label{7.3}
    \end{align*}
    where the second inequality is due to \eqref{4.5} and \eqref{5.4}.
    \end{itemize}    
    Substituting \eqref{7.2} and \eqref{7.3} into \eqref{7.1}, and taking the total expectation, we obtain
    \begin{align*}
    \mathbb{E}\left\|g^{k+1}-1_n \otimes \bar{g}^{k+1}\right\|^2
    \leq & \frac{1+\rho^2}{2} \mathbb{E}\left\|g^k-1_n \otimes \bar{g}^{k}\right\|^2+\left(9 d L^2 \rho^2+\frac{6 \rho^2 L^2}{1-\rho^2}\right) \mathbb{E}\left\|x^k-x^{k-1}\right\|^2 \\
    + & \left(3 \beta^2 \rho^2+\frac{6 \rho^2 \beta^2}{1-\rho^2}\right) \mathbb{E}\left\|m^{k-1}-\nabla f\left(x^{k-1}\right)\right\|^2\\
    + &54 \beta^2 d L^2 \rho^2\left(1+\sigma_0^2\right) \mathbb{E}\left\|x^{k-1}-1_n \otimes \bar{x}^{k-1}\right\|^2 \\
    + & 54 \beta^2 n d \rho^2\left(1+\sigma_0^2\right)\left(1+\sigma_2^2\right) \mathbb{E}\left\|\nabla f\left(\bar{x}^{k-1}\right)\right\|^2+\left(8 n d^2 L^2 \rho^2+\frac{6 \rho^2 \beta^2 L^2}{1-\rho^2}\right) t_k^2 \\
    + & 54 \beta^2 n d \rho^2\left(1+\sigma_0^2\right) \sigma_3^2+12 \beta^2 n d \rho^2 \sigma_1^2.
    \end{align*}
    Furthermore, it can be seen that
    \begin{align*}
    \left\|x^k-x^{k-1}\right\|^2 & \leq 3\left\|x^k-1_n \otimes \bar{x}^k\right\|^2+3 \alpha^2 n\left\|\bar{m}^{k-1}\right\|^2+3\left\|x^{k-1}-1_n \otimes \bar{x}^{k-1}\right\|^2 \\
    & \leq 6 \rho^2 \alpha^2\left\|g^k-1_n \otimes \bar{g}^{k}\right\|^2+3 \alpha^2 n\left\|\bar{m}^{k-1}\right\|^2+9\left\|x^{k-1}-1_n \otimes \bar{x}^{k-1}\right\|^2.
    \end{align*}
    Thus, under the conditions $\alpha \leq \frac{1-\rho^2}{\sqrt{360d} L \rho^2}$ and $\beta \leq \frac{1}{\sqrt{54\left(1-\rho^2\right)\left(1+\sigma_0^2\right)}}$, we have
    \begin{align*}
    \mathbb{E}\left\|g^{k+1}-1_n \otimes \bar{g}^{k+1}\right\|^2 
    \leq &\; \left(\frac{1+\rho^2}{2}+\frac{90 \alpha^2 d L^2 \rho^4}{1-\rho^2}\right) \mathbb{E}\left\|g^k-1_n \otimes \bar{g}^k\right\|^2+\frac{45 \alpha^2 n d L^2 \rho^2}{1-\rho^2} \mathbb{E}\left\|\bar{m}^{k-1}\right\|^2 \\
    & +\left(3 \beta^2 \rho^2+\frac{6 \rho^2 \beta^2}{1-\rho^2}\right) \mathbb{E}\left\|m^{k-1}-\nabla f\left(x^{k-1}\right)\right\|^2\\
    &+\left(54 \beta^2 d L^2 \rho^2\left(1+\sigma_0^2\right)+\frac{135 d L^2 \rho^2}{1-\rho^2}\right) \mathbb{E}\left\|x^{k-1}-1_n \otimes \bar{x}^{k-1}\right\|^2 \\
    & +54 \beta^2 n d \rho^2\left(1+\sigma_0^2\right)\left(1+\sigma_2^2\right) \mathbb{E}\left\|\nabla f\left(\bar{x}^{k-1}\right)\right\|^2+\left(8 n d^2 L^2 \rho^2+\frac{6 \rho^2 \beta^2 L^2}{1-\rho^2}\right) t_k^2 \\
    & +54 \beta^2 n d \rho^2\left(1+\sigma_0^2\right) \sigma_3^2+12 \beta^2 n d \rho^2 \sigma_1^2 \\
    \leq &\; \frac{3+\rho^2}{4} \mathbb{E}\left\|g^k-1_n \otimes \bar{g}^k\right\|^2+\frac{45 \alpha^2 n d L^2 \rho^2}{1-\rho^2} \mathbb{E}\left\|\bar{m}^{k-1}\right\|^2 \\
    & +\frac{9 \beta^2 \rho^2}{1-\rho^2} \mathbb{E}\left\|m^{k-1}-\nabla f\left(x^{k-1}\right)\right\|^2+\frac{136 d L^2 \rho^2}{1-\rho^2} \mathbb{E}\left\|x^{k-1}-1_n \otimes \bar{x}^{k-1}\right\|^2 \\
    & +54 \beta^2 n d \rho^2\left(1+\sigma_0^2\right)\left(1+\sigma_2^2\right) \mathbb{E}\left\|\nabla f\left(\bar{x}^{k-1}\right)\right\|^2+\left(8 n d^2 L^2 \rho^2+\frac{6 \rho^2 \beta^2 L^2}{1-\rho^2}\right) t_k^2 \\
    & +54 \beta^2 n d \rho^2\left(1+\sigma_0^2\right) \sigma_3^2+12 \beta^2 n d \rho^2 \sigma_1^2. 
    \end{align*}
Thus we complete the proof.
\end{proof}

\begin{lemma}
\label{lemma8}
    If $V_{k} \leq q V_{k-1}+R_{k-1}+C, \forall k \geq 1$, where $q \in (0,1)$ then we have 
    $$
    \sum_{k=0}^K V_k \leq \frac{V_0}{1-q}+\frac{1}{1-q} \sum_{k=0}^{K-1} R_k+\frac{C K}{1-q}.
    $$
\end{lemma}

\begin{proof}
    For $k \ge 1$,
    \begin{align*}
    V_k & \leq q V_{k-1}+R_{k-1}+C \\
    & \leq q^2 V_{k-2}+q R_{k-2}+R_{k-1}+q C+C \\
    & \cdots \\
    & \leq q^k V_0+\sum_{i=0}^{k-1} q^{k-1-i} R_i+C \sum_{i=0}^{k-1} q^i.
    \end{align*}
    Then we have
    \begin{align*}
    \sum_{k=0}^K V_k & \leq V_0 \sum_{k=0}^K q^k+\sum_{k=1}^K \sum_{i=0}^{k-1} q^{k-1-i} R_i+c \sum_{k=1}^K \sum_{i=0}^{k-1} q^i \\
    & \leq V_0 \sum_{k=0}^{\infty} q^k+\sum_{k=0}^{K-1}\left(\sum_{i=0}^{\infty} q^i\right) R_k+c \sum_{k=1}^K \sum_{i=0}^{\infty} q^i \\
    & =\frac{V_0}{1-q}+\frac{1}{1-q} \sum_{k=0}^{K-1} R_k+\frac{C K}{1-q}.
    \end{align*}

\vspace{-30pt}
\end{proof}

\begin{lemma}
\label{lemma9} 
For any $K\geq 0$, one has
    \begin{align*}
    \sum_{k=0}^K \mathbb{E}\left\|\bar{m}^k-\nabla \bar{f}^k\right\|^2 \leq & \frac{2}{\beta} \mathbb{E}\left\|\bar{m}^0-\nabla \bar{f}^0\right\|^2+\left(\frac{64 \beta d}{n}\left(1+\sigma_0^2\right)\left(1+\sigma_2^2\right)+48 \beta\right) \sum_{k=0}^K \mathbb{E}\left\|\nabla f\left(\bar{x}^k\right)\right\|^2 \\
    & +\frac{72 \alpha^2(1-\beta)^2 d L^2}{n\beta} \sum_{k=0}^{K-1} \mathbb{E}\left\|\bar{m}^k\right\|^2 \\
    & +\left(\frac{32 \beta d L^2\left(1+\sigma_0^2\right)}{n^2}+\frac{48 \beta L^2}{n} +\frac{144(1-\beta)^2 d L^2}{n^2 \beta} \right) \sum_{k=0}^K \mathbb{E}\left\|x^k-1_n \otimes \bar{x}^k\right\|^2 \\
    & +\left(4 \beta L^2 d^2+\frac{32(1-\beta)^2 L^2}{\beta}+\frac{12(1-\beta)^2 d^2 L^2}{n\beta}+ \frac{4(1-\beta)^2(2-\beta) L^2}{\beta^2}\right) \sum_{k=0}^K t_k^2 \\
    & +\frac{16 \beta d}{n} K \sigma_1^2+\frac{64 \beta d}{n} K\left(1+\sigma_0^2\right) \sigma_3^2\numberthis\label{9.1}
    \end{align*}
    and
    \begin{align*}
    \sum_{k=0}^K \mathbb{E}\left\|m^k-\nabla f\left(x^k\right)\right\|^2 \leq & \frac{2}{\beta} \mathbb{E}\left\|m^0-\nabla f\left(x^0\right)\right\|^2+72 \beta n d\left(1+\sigma_0^2\right)\left(1+\sigma_2^2\right) \sum_{k=0}^K \mathbb{E}\left\|\nabla f\left(\bar{x}^k\right)\right\|^2 \\
    & +\left(72 \beta d L^2\left(1+\sigma_0^2\right)+\frac{192(1-\beta)^2 d L^2}{\beta}\right) \sum_{k=0}^K \mathbb{E}\left\|x^k-1_n \otimes \bar{x}^k\right\|^2 \\
    & +\frac{96 \alpha^2(1-\beta)^2 d L^2 n}{\beta} \sum_{k=0}^{K-1} \mathbb{E}\left\|\bar{m}^k\right\|^2+72 \beta n d K\left(1+\sigma_0^2\right) \sigma_3^2+16 \beta n d K \sigma_1^2 \\
    & +\left(4 \beta n d^2 L^2+\frac{12(1-\beta)^2 n d^2 L^2}{\beta}+\frac{16 n L^2(1-\beta)}{\beta^2(2-\beta)}\right) \sum_{k=0}^K t_k^2.\numberthis\label{9.2}
    \end{align*}
\end{lemma}

\begin{proof}
    Notice that $$
    \frac{1}{1-\frac{1+(1-\beta)^2}{2}}=\frac{2}{1-(1-\beta)^2}=\frac{2}{\beta(2-\beta)}<\frac{2}{\beta}.
    $$
    Then applying Lemma \ref{lemma8} to Lemma \ref{lemma4} leads to \eqref{9.1}. Similarly, applying Lemma \ref{lemma8} to Lemma \ref{lemma5}, we have \eqref{9.2}.
\end{proof}

\begin{lemma}
\label{lemma10}
For any $K\geq 1$, one has
    \begin{align*}
    \sum_{k=1}^K \mathbb{E}\left\|g^k-1_n \otimes \bar{g}^k\right\|^2 
    \leq &\; \frac{4}{1-\rho^2} \mathbb{E}\left\|g^1-1_n \otimes \bar{g}^1\right\|^2+\frac{72 \beta \rho^2}{\left(1-\rho^2\right)^2} \mathbb{E}\left\|m^0-\nabla f\left(x^0\right)\right\|^2 \\
    & +\frac{10052 d L^2 \rho^2}{\left(1-\rho^2\right)^2} \sum_{k=0}^{K-2} \mathbb{E}\left\|x^k-1_n \otimes \bar{x}^k\right\|^2+\frac{3700 \alpha^2 n d L^2 \rho^2}{\left(1-\rho^2\right)^2} \sum_{k=0}^{K-2} \mathbb{E}\left\|\bar{m}^{k}\right\|^2 \\
    & +\frac{2808 \beta^2 n d \rho^2\left(1+\sigma_0^2\right)\left(1+\sigma_2^2\right)}{\left(1-\rho^2\right)^2} \sum_{k=0}^{K-2} \mathbb{E}\left\|\nabla f\left(\bar{x}^k\right)\right\|^2+\frac{1208 n d^2 L^2 \rho^2}{\left(1-\rho^2\right)^2} \sum_{k=0}^{K-1} t_k^2 \\
    & +\left(\frac{216 \beta^2 n d \rho^2}{1-\rho^2}+\frac{2600 \beta^3 n d \rho^2}{\left(1-\rho^2\right)^2}\right) K\left(1+\sigma_0^2\right) \sigma_3^2+\left(\frac{48 \beta^2 n d \rho^2}{1-\rho^2}+\frac{576 \beta^3 n d \rho^2}{\left(1-\rho^2\right)^2}\right) K \sigma_1^2. \numberthis\label{10.1}
    \end{align*}
\end{lemma}

\begin{proof}
    Combining Lemma \ref{lemma7} and Lemma \ref{lemma8} together, we have
    \begin{align*}
    \sum_{k=1}^K E\left\|g^k-1_n \otimes \bar{g}^{k}\right\|^2 
    \leq &\; \frac{4}{1-\rho^2} \mathbb{E}\left\|g^1-1_n \otimes \bar{g}^1\right\|^2+\frac{36 \beta^2 \rho^2}{\left(1-\rho^2\right)^2} \sum_{k=0}^{K-2} \mathbb{E}\left\|m^k-\nabla f\left(x^k\right)\right\|^2 \\
    & +\frac{180 \alpha^2 n d L^2 \rho^2}{\left(1-\rho^2\right)^2} \sum_{k=0}^{K-2} \mathbb{E}\left\|\bar{m}^k\right\|^2+\frac{544 d L^2 \rho^2}{\left(1-\rho^2\right)^2} \sum_{k=0}^{K-2} \mathbb{E}\left\|x^k-1_n \otimes \bar{x}^k\right\|^2 \\
    & +\frac{216 \beta^2 n d \rho^2\left(1+\sigma_0^2\right)\left(1+\sigma_2^2\right)}{1-\rho^2} \sum_{k=0}^{K-2} \mathbb{E}\left\|\nabla f\left(\bar{x}^k\right)\right\|^2+\left(\frac{32 n d^2 L^2 \rho^2}{1-\rho^2}+\frac{24 \beta^2 \rho^2 L^2}{\left(1-\rho^2\right)^2}\right) \sum_{k=0}^{K-1} t_k^2 \\
    & +\frac{216 \beta^2 n d \rho^2}{1-\rho^2} K\left(1+\sigma_0^2\right) \sigma_3^2+\frac{48 \beta^2 n d \rho^2}{1-\rho^2} K \sigma_1^2.
    \end{align*}
    Substituting equation \eqref{9.2} into the above equation yields
    \begin{align*}
    \sum_{k=1}^K \mathbb{E}\left\|g^k-1_n \otimes \bar{g}^k\right\|^2 
    \leq &\; \frac{4}{1-\rho^2} \mathbb{E}\left\|g^1-1_n \otimes \bar{g}^1\right\|^2+\frac{72 \beta \rho^2}{\left(1-\rho^2\right)^2} \mathbb{E}\left\|m^0-\nabla f\left(x^0\right)\right\|^2 \\
    & +\left(\frac{544 d L^2 \rho^2}{\left(1-\rho^2\right)^2} + \frac{2592 \beta^3 d L^2 \rho^2\left(1+\sigma_0^2\right)}{\left(1-\rho^2\right)^2}+\frac{6912 \beta\left(1-\beta\right)^2 d L^2 \rho^2}{\left(1-\rho^2\right)^2}\right) \sum_{k=0}^{K-2} \mathbb{E}\left\|x^k-1_n \otimes \bar{x}^k\right\|^2 \\
    & +\left(\frac{180 \alpha^2 n d L^2 \rho^2}{\left(1-\rho^2\right)^2}+\frac{3456 \alpha^2 \beta(1-\beta)^2 n d L^2 \rho^2}{\left(1-\rho^2\right)^2}\right) \sum_{k=0}^{K-2} \mathbb{E}\left\|\bar{m}^k\right\|^2 \\
    & +\left(\frac{216 \beta^2 n d \rho^2\left(1+\sigma_0^2\right)\left(1+\sigma_2^2\right)}{1-\rho^2}+\frac{2592 \beta^3 n d \rho^2\left(1+\sigma_0^2\right)\left(1+\sigma_2^2\right)}{\left(1-\rho^2\right)^2}\right) \sum_{k=0}^{K-2} \mathbb{E}\left\|\nabla f\left(\bar{x}^k\right)\right\|^2 \\
    & +\left(\frac{32 n d^2 L^2 \rho^2}{1-\rho^2}+\frac{24 \beta^2 \rho^2 L^2}{\left(1-\rho^2\right)^2}+\frac{144 \beta^3 n d^2 L^2 \rho^2}{\left(1-\rho^2\right)^2} \right. \\
    & \quad \quad \left. +\frac{432 \beta(1-\beta)^2 n d^2 L^2 \rho^2}{\left(1-\rho^2\right)^2}+\frac{576 n L^2(1-\beta) \rho^2}{\left(1-\rho^2\right)^2}\right) \sum_{k=0}^{K-1} t_k^2 \\
    & +\left(\frac{216 \beta^2 n d \rho^2}{1-\rho^2}+\frac{2592 \beta^3 n d \rho^2}{\left(1-\rho^2\right)^2}\right) K\left(1+\sigma_0^2\right) \sigma_3^2+\left(\frac{48 \beta^2 n d \rho^2}{1-\rho^2}+\frac{576 \beta^3 n d \rho^2}{\left(1-\rho^2\right)^2}\right) K \sigma_1^2 \\
    \leq &\; \frac{4}{1-\rho^2} \mathbb{E}\left\|g^1-1_n \otimes \bar{g}^1\right\|^2+\frac{72 \beta \rho^2}{\left(1-\rho^2\right)^2} \mathbb{E}\left\|m^0-\nabla f\left(x^0\right)\right\|^2 \\
    & +\frac{10052 d L^2 \rho^2}{\left(1-\rho^2\right)^2} \sum_{k=0}^{K-2} \mathbb{E}\left\|x^k-1_n \otimes \bar{x}^k\right\|^2+\frac{3700 \alpha^2 n d L^2 \rho^2}{\left(1-\rho^2\right)^2} \sum_{k=0}^{K-2} \mathbb{E}\left\|\bar{m}^{k}\right\|^2 \\
    & +\frac{2808 \beta^2 n d \rho^2\left(1+\sigma_0^2\right)\left(1+\sigma_2^2\right)}{\left(1-\rho^2\right)^2} \sum_{k=0}^{K-2} \mathbb{E}\left\|\nabla f\left(\bar{x}^k\right)\right\|^2+\frac{1208 n d^2 L^2 \rho^2}{\left(1-\rho^2\right)^2} \sum_{k=0}^{K-1} t_k^2 \\
    & +\left(\frac{216 \beta^2 n d \rho^2}{1-\rho^2}+\frac{2600 \beta^3 n d \rho^2}{\left(1-\rho^2\right)^2}\right) K\left(1+\sigma_0^2\right) \sigma_3^2+\left(\frac{48 \beta^2 n d \rho^2}{1-\rho^2}+\frac{576 \beta^3 n d \rho^2}{\left(1-\rho^2\right)^2}\right) K \sigma_1^2,
    \end{align*}
which completes the proof.
\end{proof}

\begin{lemma}
\label{lemma11}
 Suppose $\alpha \leq \frac{\left(1-\rho^2\right)^2}{284 \sqrt{d} L \rho^2}$. One has
    \begin{align*}
    & \sum_{k=0}^K\left\|x^k-1_n \otimes \bar{x}^k\right\|^2 \\
    \leq &\; \frac{32 \alpha^2 \rho^2}{\left(1-\rho^2\right)^3} \mathbb{E}\left\|g^1-1_n \otimes \bar{g}^1\right\|^2+\frac{576 \alpha^2 \beta \rho^4}{\left(1-\rho^2\right)^4} \mathbb{E}\left\|m^0-\nabla f\left(x^0\right)\right\|^2 +\frac{29600 \alpha^4 n d L^2 \rho^4}{\left(1-\rho^2\right)^4} \sum_{k=0}^{K-2} \mathbb{E}\left\|\bar{m}^k\right\|^2 \\
    & +\frac{22600 \alpha^2 \beta^2 n d \rho^4\left(1+\sigma_0^2\right)\left(1+\sigma_2^2\right)}{\left(1-\rho^2\right)^4} \sum_{k=0}^{K-2} \mathbb{E}\left\|\nabla f\left(\bar{x}^k\right)\right\|^2 +\frac{10000 \alpha^2 n d^2 L^2 \rho^4}{\left(1-\rho^2\right)^4} \sum_{k=0}^{K-1} t_k^2 \\
    & +\left(\frac{1800 \alpha^2 \beta^2 n d \rho^4}{\left(1-\rho^2\right)^3}+\frac{20800 \alpha^2 \beta^3 n d \rho^4}{\left(1-\rho^2\right)^4}\right) K\left(1+\sigma_0^2\right) \sigma_3^2+\left(\frac{400 \alpha^2 \beta^2 n d \rho^4}{\left(1-\rho^2\right)^3}+\frac{4700 \alpha^2 \beta^3 n d \rho^4}{\left(1-\rho^2\right)^4}\right) K \sigma_1^2. \numberthis\label{11.1}
    \end{align*}
\end{lemma}
    
\begin{proof}
Applying Lemma \ref{lemma8} to Lemma \ref{lemma6}, we have
    $$
    \sum_{k=0}^K \left\|x^k-1_n \otimes \bar{x}^k\right\|^2 \leq \frac{4 \alpha^2 \rho^2}{\left(1-\rho^2\right)^2} \sum_{k=1}^K \left\|g^k-1_n \otimes \bar{g}^k\right\|^2.
    $$
    Moreover, plugging \eqref{10.1} into the above equation gives that
    \begin{align*}
    & \sum_{k=0}^K\left\|x^k-1_n \otimes \bar{x}^k\right\|^2 \\
    \leq &\; \frac{4 \alpha^2 \rho^2}{\left(1-\rho^2\right)^2} \sum_{k=1}^K\left\|g^k-1_n \otimes \bar{g}^{k}\right\|^2 \\
    \leq &\; \frac{16 \alpha^2 \rho^2}{\left(1-\rho^2\right)^3} \mathbb{E}\left\|g^1-1_n \otimes \bar{g}^1\right\|^2+\frac{288 \alpha^2 \beta \rho^4}{\left(1-\rho^2\right)^4} \mathbb{E}\left\|m^0-\nabla f\left(x^0\right)\right\|^2 \\
    & +\frac{40208 \alpha^2 d L^2 \rho^4}{\left(1-\rho^2\right)^4} \sum_{k=0}^K \mathbb{E}\left\|x^k-1_n \otimes \bar{x}^k\right\|^2+\frac{14800 \alpha^4 n d L^2 \rho^4}{\left(1-\rho^2\right)^4} \sum_{k=0}^{K-2} \mathbb{E}\left\|\bar{m}^k\right\|^2 \\
    & +\frac{11300 \alpha^2 \beta^2 n d \rho^4\left(1+\sigma_0^2\right)\left(1+\sigma_2^2\right)}{\left(1-\rho^2\right)^4} \sum_{k=0}^{K-2} \mathbb{E}\left\|\nabla f\left(\bar{x}^k\right)\right\|^2+\frac{4832 \alpha^2 n d^2 L^2 \rho^4}{\left(1-\rho^2\right)^4} \sum_{k=0}^{K-1} t_k^2 \\
    & +\left(\frac{864 \alpha^2 \beta^2 n d \rho^4}{\left(1-\rho^2\right)^3}+\frac{10400 \alpha^2 \beta^3 n d \rho^4}{\left(1-\rho^2\right)^4}\right) K\left(1+\sigma_0^2\right) \sigma_3^2+\left(\frac{192 \alpha^2 \beta^2 n d \rho^4}{\left(1-\rho^2\right)^3}+\frac{2304 \alpha^2 \beta^3 n d \rho^4}{\left(1-\rho^2\right)^4}\right) K \sigma_1^2.
    \end{align*}
    It is easy to see that $1-\frac{40208 \alpha^2 d L^2 \rho^4}{\left(1-\rho^2\right)^4} \ge \frac{1}{2}$ provided $\alpha \leq \frac{\left(1-\rho^2\right)^2}{284 \sqrt{d} L \rho^2}$.
    Thus we have
    \begin{align*}
    & \sum_{k=0}^K\left\|x^k-1_n \otimes \bar{x}^k\right\|^2 \\
    \leq &\; \frac{32 \alpha^2 \rho^2}{\left(1-\rho^2\right)^3} \mathbb{E}\left\|g^1-1_n \otimes \bar{g}^1\right\|^2+\frac{576 \alpha^2 \beta \rho^4}{\left(1-\rho^2\right)^4} \mathbb{E}\left\|m^0-\nabla f\left(x^0\right)\right\|^2 +\frac{29600 \alpha^4 n d L^2 \rho^4}{\left(1-\rho^2\right)^4} \sum_{k=0}^{K-2} \mathbb{E}\left\|\bar{m}^k\right\|^2 \\
    & +\frac{22600 \alpha^2 \beta^2 n d \rho^4\left(1+\sigma_0^2\right)\left(1+\sigma_2^2\right)}{\left(1-\rho^2\right)^4} \sum_{k=0}^{K-2} \mathbb{E}\left\|\nabla f\left(\bar{x}^k\right)\right\|^2 +\frac{10000 \alpha^2 n d^2 L^2 \rho^4}{\left(1-\rho^2\right)^4} \sum_{k=0}^{K-1} t_k^2 \\
    & +\left(\frac{1800 \alpha^2 \beta^2 n d \rho^4}{\left(1-\rho^2\right)^3}+\frac{20800 \alpha^2 \beta^3 n d \rho^4}{\left(1-\rho^2\right)^4}\right) K\left(1+\sigma_0^2\right) \sigma_3^2+\left(\frac{400 \alpha^2 \beta^2 n d \rho^4}{\left(1-\rho^2\right)^3}+\frac{4700 \alpha^2 \beta^3 n d \rho^4}{\left(1-\rho^2\right)^4}\right) K \sigma_1^2,
    \end{align*}
    which completes the proof.
\end{proof}

\begin{lemma}
\label{lemma12}
Consider the gradient estimation $\bar{m}^0$ generated by Algorithm \ref{algo}. Under Assumption \ref{ass2} and Assumption \ref{ass3}, one has
\begin{align*}
\mathbb{E}\left\|\bar{m}^0-\nabla \bar{f}^0\right\|^2
& \leq \frac{24 d\left(1+\sigma_0^2\right)+6}{n b_0}\left\|\nabla f\left(x^0\right)\right\|^2 +\frac{24 d \sigma_1^2}{b_0}+\frac{3 d^2 L^2 }{b_0}t_0^2 +\frac{6L^2}{b_0} t_0^2 +2 L^2 t_0^2 . 
\end{align*}
\end{lemma}

\begin{proof}
Recall that
\begin{align*}
    \bar{m}^0 = \frac{1}{n}\sum_{i=1}^n m_i^0 = \frac{1}{n}\sum_{i=1}^n \left(\frac{1}{b_0}\sum_{s=1}^{b_0}\hat{\nabla} f_i(x_i^0; u_{i,s}^0, \xi_{i,s}^0)\right) 
\end{align*}
and $\mathbb{E}_{u_{i,s}^0, \xi_{i,s}^0}\left[ \hat{\nabla} f_i(x_i^0; u_{i,s}^0, \xi_{i,s}^0) \right] = \nabla f_{i,t_0}(x_i^0)$. Then a direct computation yields that
\begin{align*}
& \mathbb{E}\left\|\bar{m}^0-\nabla \bar{f}^0\right\|^2\\
= &\; \mathbb{E}\left\|\frac{1}{n} \sum_{i=1}^n \left(m_i^0-\nabla f_i\left(x_i^0\right) \right)\right\|^2 \\
\leq &\; \frac{1}{n} \sum_{i=1}^n \mathbb{E}\left\|m_i^0-\nabla f_i\left(x_i^0\right)\right\|^2 \\
= &\; \frac{1}{n} \sum_{i=1}^n \mathbb{E}\left\|\frac{1}{b_0} \sum_{s=1}^{b_0} \hat{\nabla} f_i\left(x_i^0 ; u_{i, s}^0, \xi_{i, s}^0\right)-\nabla f_i\left(x_i^0\right)\right\|^2 \\
\leq &\; \frac{2}{n} \sum_{i=1}^n \mathbb{E}\left\|\frac{1}{b_0} \sum_{s=1}^{b_0} \hat{\nabla} f_i\left(x_i^0 ; u_{i, s}^0, \xi_{i, s}^0\right)-\nabla f_{i,t_0}\left(x_i^0\right)\right\|^2 + \frac{2}{n} \sum_{i=1}^n \left\|\nabla f_{i,t_0}\left(x_i^0\right) - \nabla f_{i}(x_i^0)\right\|^2 \\
\stackrel{(a)}{\leq} &\;\frac{2}{n} \sum_{i=1}^n \mathbb{E}\left\|\frac{1}{b_0} \sum_{s=1}^{b_0} \hat{\nabla} f_i\left(x_i^0 ; u_{i, s}^0, \xi_{i, s}^0\right)-\nabla f_{i,t_0}\left(x_i^0\right)\right\|^2 + 2t_0^2L^2\\
\stackrel{(b)}{=}&\;\frac{2}{n} \sum_{i=1}^n \frac{1}{b_0^2} \sum_{s=1}^{b_0} \mathbb{E}\left\|\hat{\nabla} f_i\left(x_i^0 ; u_{i, s}^0, \xi_{i, s}^0\right)-\nabla f_{i,t_0}\left(x_i^0\right)\right\|^2 + 2t_0^2L^2\\
\leq &\; \frac{6}{n b_0^2} \sum_{i=1}^n \sum_{s=1}^{b_0} \mathbb{E}\left\|\hat{\nabla} f_i\left(x_i^0 ; u_{i, s}^0, \xi_{i, s}^0\right)\right\|^2+\frac{6}{n b_0^2} \sum_{i=1}^n \sum_{s=1}^{b_0} \mathbb{E}\left\|\nabla f_{i, t_0}\left(x_i^0\right)-\nabla f_i\left(x_i^0\right)\right\|^2 \\
& +\frac{6}{n b_0^2} \sum_{i=1}^n \sum_{S=1}^{b_0}\left\|\nabla f_i\left(x_i^0\right)\right\|^2+2 L^2 t_0^2 \\
\stackrel{(c)}{\leq}&\; \frac{6}{n b_0^2} \sum_{i=1}^n \sum_{s=1}^{b_0}\left(4 d\left(1+\sigma_0^2\right)\left\|\nabla f_i\left(x_i^0\right)\right\|^2+\frac{1}{2} d^2 L^2 t_0^2+4 d \sigma_1^2\right)\\
&\;+\frac{6}{b_0} L^2 t_0^2 +\frac{6}{n b_0}\left\|\nabla f\left(x^0\right)\right\|^2+2 L^2 t_0^2\\
= &\; \frac{24 d\left(1+\sigma_0^2\right)+6}{n b_0}\left\|\nabla f\left(x^0\right)\right\|^2 +\frac{24 d \sigma_1^2}{b_0}+\frac{3 d^2 L^2 }{b_0}t_0^2 +\frac{6L^2}{b_0} t_0^2 +2 L^2 t_0^2,
\end{align*}
where step (a) is due to Lemma \ref{lemma1}, step (b) uses the fact that 
\begin{align*}
    \mathbb{E}\left\langle \hat{\nabla} f_i\left(x_i^0 ; u_{i, s}^0, \xi_{i, s}^0\right)-\nabla f_{i,t_0}\left(x_i^0\right), \hat{\nabla} f_i\left(x_i^0 ; u_{i, s'}^0, \xi_{i, s'}^0\right)-\nabla f_{i,t_0}\left(x_i^0\right)\right\rangle = 0, 
\end{align*}
and step (c) follows from the fact that 
\begin{align*}
    \mathbb{E}_{u_i, \xi_i}\left\|\hat{\nabla} f_i\left(x_i^0 ; u_i, \xi_i\right)\right\|^2 \leq 4 d\left(1+\sigma_0^2\right)\left\|\nabla f_i\left(x_i^0\right)\right\|^2+\frac{1}{2} d^2 L_t^2 t_0^2+4 d \sigma_1^2.
\end{align*}
Moreover, the fact used in step (c) can be proved as follows:
\begin{align*}
& \mathbb{E}_{u_i, \xi_i}\left\|\hat{\nabla} f_i\left(x_i^0 ; u_i, \xi_i\right)\right\|^2 \\
= &\; E_{\xi_i} \mathbb{E}_{u_i}\left\|d \frac{F_i\left(x_i^0+t_0 u_i ; \xi_i\right)-F_i\left(x_i^0-t_0 u_i ; \xi_i\right)}{2 t_0} u_i\right\|^2 \\
\leq &\; 2 \mathbb{E}_{\xi_i} \mathbb{E}_{u_i}\left\|d \frac{F_i\left(x_i^0+t_0 u_i ; \xi_i\right)-F_i\left(x_i^0-t_0 u_i ; \xi_i\right)}{2 t_0} u_i-d \left\langle \nabla F_i\left(x_i^0 ; \xi_i\right), u_i \right\rangle u_i\right\|^2 \\
& +2 \mathbb{E}_{\xi_i} \mathbb{E}_{u_i}\left\|d \left\langle \nabla F_i\left(x_i^0 ; \xi_i\right), u_i \right\rangle u_i\right\|^2 \\
\leq &\; \frac{1}{2} d^2 L^2 t_0^2+2 d \mathbb{E}_{\xi_i}\left\|\nabla F_i\left(x_i^0 ; \xi_i\right)\right\|^2 \\
\leq &\; \frac{1}{2} d^2 L^2 t_0^2+4 d \mathbb{E}_{\xi_i}\left\|\nabla F_i\left(x_i^0 ; \xi_i\right)-\nabla f_i\left(x_i^0\right)\right\|^2+4 d\left\|\nabla f_i\left(x_i^0\right)\right\|^2 \\
\leq &\; \frac{1}{2} d^2 L^2 t_0^2+4 d\left(\sigma_0^2\left\|\nabla f_i\left(x_i^0\right)\right\|^2+\sigma_1^2\right)+4 d\left\|\nabla f_i\left(x_i^0\right)\right\|^2 \\
= &\; 4 d\left(1+\sigma_0^2\right)\left\|\nabla f_i\left(x_i^0\right)\right\|^2+\frac{1}{2} d^2 L_t^2 t_0^2+4 d \sigma_1^2,
\end{align*}
where the second inequality is due to Lemma \ref{lemma1} and the last inequality follows from Assumption \ref{ass3}. Thus we complete the proof.
\end{proof}

\section{Proof of Theorem \ref{thm}}
\label{appendix_b}
For ease of notation, define
\begin{align*}
    c_0 & = \frac{L^2}{n}+\frac{32 d L^2\left(1+\sigma_0^2\right)}{n^2}+\frac{48 L^2}{n},\\
    c_1 & = \frac{64 d\left(1+\sigma_0^2\right)\left(1+\sigma_2^2\right)}{n}+48+\frac{3277000 \alpha^2 d^2 L^2 \rho^4\left(1+\sigma_0^2\right)\left(1+\sigma_2^2\right)}{\left(1-\rho^2\right)^4},\\
    c_2 & = \frac{72 \alpha^2 d L^2}{n}+\frac{4292000 \alpha^4 d^2 L^4 \rho^4}{\left(1-\rho^2\right)^4},\\
    c_3 & = \frac{72 d L^2}{n}+\frac{1073000 d^2 L^2 \rho^4}{\left(1-\rho^2\right)^4},\\
    c_4 & = \frac{1}{\sqrt{54\left(1-\rho^2\right)\left(1+\sigma_0^2\right)}}.
\end{align*}

Since $f(\cdot)$ is $L$-smooth, we have
\begin{align*}
f\left(\bar{x}^{k+1}\right) & \leq f\left(\bar{x}^k\right)+\left\langle\nabla f\left(\bar{x}^k\right), \bar{x}^{k+1}-\bar{x}^k\right\rangle+\frac{L}{2}\left\|\bar{x}^{k+1}-\bar{x}^k\right\|^2 \\
& =f\left(\bar{x}^k\right)-\alpha\left\langle\nabla f\left(\bar{x}^k\right), \bar{m}^k \right\rangle+\frac{\alpha^2 L}{2}\left\|\bar{m}^k\right\|^2 \\
& =f\left(\bar{x}^k\right)-\alpha\left(\frac{1}{2}\left\|\nabla f\left(\bar{x}^k\right)\right\|^2+\frac{1}{2}\left\|\bar{m}^k\right\|^2-\frac{1}{2}\left\|\bar{m}^k-\nabla f\left(\bar{x}^k\right)\right\|^2\right)+\frac{\alpha^2 L}{2}\left\|\bar{m}^k\right\|^2 \\
& =f\left(\bar{x}^k\right)-\frac{\alpha}{2}\left\|\nabla f\left(\bar{x}^k\right)\right\|^2-\left(\frac{\alpha}{2}-\frac{\alpha^2 L}{2}\right)\left\|\bar{m}^k\right\|^2+\frac{\alpha}{2}\left\|\bar{m}^k-\nabla f\left(\bar{x}^k\right)\right\|^2 \\
& \leq f\left(\bar{x}^k\right)-\frac{\alpha}{2}\left\|\nabla f\left(\bar{x}^k\right)\right\|^2-\left(\frac{\alpha}{2}-\frac{\alpha^2 L}{2}\right)\left\|\bar{m}^k\right\|^2+\alpha\left\|\bar{m}^k-\nabla \bar{f}^k\right\|^2+\alpha\left\|\nabla \bar{f}^k-\nabla f\left(\bar{x}^{k}\right)\right\|^2 \\
& \leq f\left(\bar{x}^k\right)-\frac{\alpha}{2}\left\|\nabla f\left(\bar{x}^k\right)\right\|^2-\left(\frac{\alpha}{2}-\frac{\alpha^2 L}{2}\right)\left\|\bar{m}^k\right\|^2+\alpha\left\|\bar{m}^k-\nabla \bar{f}^k\right\|^2+\frac{\alpha L^2}{n}\left\|x^k-1_n \otimes \bar{x}^k\right\|^2,
\end{align*}
where the second line is due to the fact $\bar{x}^{k+1} = \bar{x}^k - \alpha \bar{m}^k$.
Then a simple calculation yields that
\begin{align*}
0 \leq &\; \mathbb{E} f\left(\bar{x}^{K+1}\right)-f^* \\
\leq &\; \mathbb{E} f\left(\bar{x}^K\right)-f^*-\frac{\alpha}{2} \mathbb{E}\left\|\nabla f\left(\bar{x}^K\right)\right\|^2-\left(\frac{\alpha}{2}-\frac{\alpha^2 L}{2}\right) \mathbb{E}\left\|\bar{m}^K\right\|^2\\
& +\alpha \mathbb{E}\left\|\bar{m}^K-\nabla \bar{f}^K\right\|^2+\frac{\alpha L^2}{n} \mathbb{E}\left\|x^K-1_n \otimes \bar{x}^K\right\|^2 \\
\leq &\; f\left(\bar{x}^0\right)-f^*-\frac{\alpha}{2} \sum_{k=0}^K \mathbb{E}\left\|\nabla f\left(\bar{x}^k\right)\right\|^2-\left(\frac{\alpha}{2}-\frac{\alpha^2 L}{2}\right) \sum_{k=0}^K\mathbb{E}\left\|\bar{m}^k\right\|^2 \\
& +\alpha \sum_{k=0}^K \mathbb{E}\left\|\bar{m}^{k}-\nabla \bar{f}^k\right\|^2+\frac{\alpha L^2}{n} \sum_{k=0}^K\mathbb{E}\left\|x^k-1_n \otimes \bar{x}^k\right\|^2. \numberthis\label{eq1}
\end{align*}
Substituting \eqref{9.1} into \eqref{eq1}, we get
\begin{align*}
0 \leq &\; f\left(\bar{x}^0\right)-f^*-\left(\frac{\alpha}{2}-\frac{64 \alpha \beta d}{n}\left(1+\sigma_0^2\right)\left(1+\sigma_2^2\right)-48 \alpha \beta\right) \sum_{k=0}^K \mathbb{E}\left\|\nabla f\left(\bar{x}^k\right)\right\|^2 \\
& -\left(\frac{\alpha}{2}-\frac{\alpha^2 L}{2}-\frac{72 \alpha^3(1-\beta)^2 d L^2}{n\beta}\right) \sum_{k=0}^K \mathbb{E}\left\|\bar{m}^k\right\|^2 +\frac{2 \alpha}{\beta} \mathbb{E}\left\|\bar{m}^0-\nabla \bar{f}^0\right\|^2 \\
& +\left(\frac{\alpha L^2}{n}+\frac{32 \alpha \beta d L^2\left(1+\sigma_0^2\right)}{n^2}+\frac{48 \alpha \beta L^2}{n}+\frac{144 \alpha(1-\beta)^2 d L^2}{n^2 \beta}\right) \sum_{k=0}^K \mathbb{E}\left\|x^k-1_n \otimes \bar{x}^k\right\|^2 \\
& +\left(4 \alpha \beta L^2 d^2+\frac{32\alpha(1-\beta)^2 L^2}{\beta}+\frac{12\alpha(1-\beta)^2 d^2 L^2}{n \beta}+\frac{4 \alpha(1-\beta)^2(2-\beta) L^2}{\beta^2}\right) \sum_{k=0}^K t_k^2 \\
& +\frac{16 \alpha \beta d}{n} K \sigma_1^2+\frac{64 \alpha \beta d}{n} K\left(1+\sigma_0^2\right) \sigma_3^2 \\
\leq &\; f\left(\bar{x}^0\right)-f^*-\left(\frac{\alpha}{2}-\frac{64 \alpha \beta d}{n}\left(1+\sigma_0^2\right)\left(1+\sigma_2^2\right)-48 \alpha \beta\right) \sum_{k=0}^K \mathbb{E}\left\|\nabla f\left(\bar{x}^k\right)\right\|^2 \\
& -\left(\frac{\alpha}{4}-\frac{72 \alpha^3(1-\beta)^2 d L^2}{n\beta}\right) \sum_{k=0}^K \mathbb{E}\left\|\bar{m}^k\right\|^2 +\frac{2 \alpha}{\beta} \mathbb{E}\left\|\bar{m}^0-\nabla \bar{f}^0\right\|^2 \\
& +\frac{145 \alpha d L^2}{n \beta} \sum_{k=0}^K \mathbb{E}\left\|x^k-1_n \otimes \bar{x}^k\right\|^2 +\frac{52 \alpha L^2 d^2}{\beta^2} \sum_{k=0}^K t_k^2+\frac{16 \alpha \beta d}{n} K \sigma_1^2+\frac{64 \alpha \beta d}{n} K\left(1+\sigma_0^2\right) \sigma_3^2, \numberthis\label{eq2}
\end{align*}
where the last line is due to $\alpha \leq \frac{1}{2L}$ and $\beta \leq \frac{L^2 d}{nc_0}$. Plugging \eqref{11.1} into \eqref{eq2} yields that
\begin{align}
0 \leq & f\left(\bar{x}^0\right)-f^*\notag\\
&-\left(\frac{\alpha}{2}-\frac{64 \alpha \beta d}{n}\left(1+\sigma_0^2\right)\left(1+\sigma_2^2\right)-48 \alpha \beta-\frac{3277000 \alpha^3 \beta d^2 L^2 \rho^4\left(1+\sigma_0^2\right)\left(1+\sigma_2^2\right)}{\left(1-\rho^2\right)^4}\right) \sum_{k=0}^K \mathbb{E}\left\|\nabla f\left(\bar{x}^k\right)\right\|^2 \notag\\
& -\left(\frac{\alpha}{4}-\frac{72 \alpha^3(1-\beta)^2 d L^2}{n\beta}-\frac{4292000 \alpha^5 d^2 L^4 \rho^4}{\beta\left(1-\rho^2\right)^4}\right) \sum_{k=0}^K E\left\|\bar{m}^k\right\|^2 \notag\\
& +\frac{2 \alpha}{\beta} \mathbb{E}\left\|\bar{m}^0-\nabla \bar{f}^0\right\|^2+\frac{4640 \alpha^3 d L^2 \rho^2}{\beta n\left(1-\rho^2\right)^3} E\left\|g^1-1_n \otimes g^{1}\right\|^2+\frac{83520 \alpha^3 d L^2 \rho^4}{n\left(1-\rho^2\right)^4} \mathbb{E}\left\|m^0-\nabla f\left(x^0\right)\right\|^2 \notag\\
& +\left(\frac{52 \alpha L^2 d^2}{\beta^2}+\frac{1450000 \alpha^3 d^3 L^4 \rho^4}{\beta\left(1-\rho^2\right)^4}\right) \sum_{k=0}^K t_k^2 \notag\\
& +\left(\frac{64 \alpha d}{n}+\frac{261000 \alpha^3 \beta d^2 L^2 \rho^4}{\left(1-\rho^2\right)^3}+\frac{3016000 \alpha^3 \beta^2 d^2 L^2 \rho^4}{\left(1-\rho^2\right)^4}\right) K\left(1+\sigma_0^2\right) \sigma_3^2\notag \\
& +\left(\frac{16 \alpha \beta d}{n}+\frac{58000 \alpha^3 \beta d^2 L^2 \rho^4}{\left(1-\rho^2\right)^3}+\frac{681500 \alpha^3 \beta^2 d^2 L^2 \rho^4}{\left(1-\rho^2\right)^4}\right) K \sigma_1^2.\label{appb 25}
\end{align}

\noindent Given the definitions of $c_1$ and $c_2$, and considering the assumption that $4c_2 \leq \beta \leq \min\{1, 1/(4c_1)\}$, we have
\begin{align*}
\dfrac{64 \alpha \beta d \left(1+\sigma_0^2\right)\left(1+\sigma_2^2\right)}{n} + 48 \alpha \beta + \dfrac{3277000 \alpha^3 \beta d^2 L^2 \rho^4\left(1+\sigma_0^2\right)\left(1+\sigma_2^2\right)}{\left(1-\rho^2\right)^4} &= \alpha \beta c_1\leq \dfrac{\alpha}{4},\\
\dfrac{72 \alpha^3(1-\beta)^2 d L^2}{n \beta} + \dfrac{4292000 \alpha^5 d^2 L^4 \rho^4}{\beta\left(1-\rho^2\right)^4} &\leq \frac{\alpha c_2}{\beta }\leq \dfrac{\alpha}{4}.
\end{align*}
Using these relations, \eqref{appb 25} can be simplified as follows:
\begin{align}
\frac{1}{K} \sum_{k=0}^{K-1} \mathbb{E} \left\|\nabla f\left(\bar{x}^k\right)\right\|^2 &
\leq\frac{4\left(f\left(\bar{x}^0\right)-f^*\right)}{\alpha K}+\frac{8}{\beta K} \mathbb{E}\left\|\bar{m}^0-\nabla \bar{f}^0\right\|^2 \notag\\
& +\frac{18560 \alpha^2 d L^2 \rho^2}{K \beta n\left(1-\rho^2\right)^3} \mathbb{E}\left\|g^1-1_n \otimes \bar{g}^{1}\right\|^2+\frac{335000 \alpha^2 d L^2 \rho^4}{K n\left(1-\rho^2\right)^4} \mathbb{E}\left\|m^0-\nabla f\left(x^0\right)\right\|^2 \notag\\
& +\frac{1}{K}\left(\frac{208 L^2 d^2}{\beta^2}+\frac{5800000 \alpha^2 d^3 L^4 \rho^4}{\beta\left(1-\rho^2\right)^4}\right) \sum_{k=0}^{K-1} t_k^2 \notag\\
& +\left(\frac{256 \beta d}{n}+\frac{1044000 \alpha^2 \beta d^2 L^2 \rho^4}{\left(1-\rho^2\right)^3}+\frac{12064000 \alpha^2 \beta^2 d^2 L^2 \rho^4}{\left(1-\rho^2\right)^4}\right)\left(1+\sigma_0^2\right) \sigma_3^2 \notag \\
& +\left(\frac{64 \beta d}{n}+\frac{232000 \alpha^2 \beta d^2 L^2 \rho^4}{\left(1-\rho^2\right)^3}+\frac{2726000 \alpha^2 \beta^2 d^2 L^2 \rho^4}{\left(1-\rho^2\right)^4}\right) \sigma_1^2.\label{eq 26}
\end{align}
Moreover, it can be seen that
\begin{align*}
\mathbb{E}\left\|m^0-\nabla f\left(x^0\right)\right\|^2 & =\sum_{i=1}^n \mathbb{E}\left\|m_i^0-\nabla f_i\left(x_i^0\right)\right\|^2 \\
&=n \mathbb{E}\left\|\bar{m}^0-\nabla \bar{f}^0\right\|^2 \\
& \leq \frac{24 d\left(1+\sigma_0^2\right)+6}{b_0}\left\|\nabla f\left(x^0\right)\right\|^2+\frac{24 n d \sigma_1^2}{b_0} +\frac{3 n d^2 L^2 }{b_0} t_0^2+\frac{6 n L^2}{b_0} t_0^2 +2 n L^2 t_0^2:=e_m,\\
\mathbb{E}\left\|g^1-1_n \otimes \bar{g}^{1}\right\|^2 & =\mathbb{E}\left\|\left(W-\frac{1_n 1_n^T}{n}\right) \otimes I_d\left(g^0+m^0-m^{-1}\right)\right\|^2 \\
& \leq \rho^2 \mathbb{E}\left\|m^0\right\|^2 \\
& \leq 2 \rho^2 \mathbb{E}\left\|m^0-\nabla f\left(x^0\right)\right\|^2+2 \rho^2\left\|\nabla f\left(x^0\right)\right\|^2\\
&=2\rho^2 e_m+2\rho^2\left\|\nabla f\left(x^0\right)\right\|^2:=e_g,
\end{align*}
where the first inequality follows from Lemma \ref{lemma12}. Applying these relationships and Lemma \ref{lemma12}, we finally have
\begin{align*}
\frac{1}{K} \sum_{k=0}^{K-1} \mathbb{E}\left\|\nabla f\left(\bar{x}^k\right)\right\|^2 \leq & \frac{4\left(f\left(\bar{x}^0\right)-f^*\right)}{\alpha K}+ \frac{\left(192 d\left(1+\sigma_0^2\right)+48\right)\left\|\nabla f\left(x^0\right)\right\|^2}{n b_0 \beta K}+\frac{192 d \sigma_1^2}{b_0 \beta K}\\
& +\left(\frac{24 d^2 L^2}{b_0 \beta K}+\frac{48 L^2}{b_0 \beta K}+\frac{16 L^2}{\beta K}\right) t_0^2 +\frac{18560 \alpha^2 d L^2 \rho^2}{K \beta n\left(1-\rho^2\right)^3} e_g +\frac{335000 \alpha^2 d L^2 \rho^4}{K n\left(1-\rho^2\right)^4} e_m \\
& +\frac{1}{K}\left(\frac{208 L^2 d^2}{\beta^2}+\frac{5800000 \alpha^2 d^3 L^4 \rho^4}{\beta\left(1-\rho^2\right)^4}\right) \sum_{k=0}^{K-1} t_k^2 \\
& +\left(\frac{256 \beta d}{n}+\frac{1044000 \alpha^2 \beta d^2 L^2 \rho^4}{\left(1-\rho^2\right)^3}+\frac{12064000 \alpha^2 \beta^2 d^2 L^2 \rho^4}{\left(1-\rho^2\right)^4}\right)\left(1+\sigma_0^2\right) \sigma_3^2 \\
& +\left(\frac{64 \beta d}{n}+\frac{232000 \alpha^2 \beta d^2 L^2 \rho^4}{\left(1-\rho^2\right)^3}+\frac{2726000 \alpha^2 \beta^2 d^2 L^2 \rho^4}{\left(1-\rho^2\right)^4}\right) \sigma_1^2.
\end{align*}
The proof of Theorem \ref{thm} is complete after selecting $\{t_k\}_{k=0,\cdots,K}$ such that $ t_0\le \frac{\beta}{d^2}$ and $ \sum\limits_{k=0}^K t_k^2 \le \frac{\beta^2 M_t}{d^4}$.
\end{document}